\documentclass[11pt]{amsart}
\usepackage{amsmath}
\usepackage{amssymb}
\usepackage{amscd}
\def\NZQ{\mathbb}               % the font for N,Z,Q,R,C
\def\NN{{\NZQ N}}
\def\QQ{{\NZQ Q}}
\def\ZZ{{\NZQ Z}}
\def\RR{{\NZQ R}}

\newtheorem{Theorem}{Theorem}[section]
\newtheorem{Lemma}[Theorem]{Lemma}
\newtheorem{Corollary}[Theorem]{Corollary}
\newtheorem{Proposition}[Theorem]{Proposition}
\newtheorem{Remark}[Theorem]{Remark}

\newtheorem{Definition}[Theorem]{Definition}

\let\epsilon\varepsilon
\let\phi=\varphi
\let\kappa=\varkappa

\begin{document}
\title{Mixed multiplicities of divisorial filtrations}
\author{Steven Dale Cutkosky}

\thanks{The first author was partially supported by NSF grant DMS-1700046.}

\address{Steven Dale Cutkosky, Department of Mathematics,
University of Missouri, Columbia, MO 65211, USA}
\email{cutkoskys@missouri.edu}

\begin{abstract}
Suppose that $R$ is an excellent local domain with maximal ideal $m_R$. 
The theory of multiplicities and mixed multiplicities of $m_R$-primary ideals extends to (possibly non Noetherian) filtrations of $R$ by $m_R$-primary ideals, and many of the classical theorems for $m_R$-primary ideals continue to hold for filtrations. The celebrated theorems involving inequalities continue to hold for filtrations, but the good conclusions that hold in the case of equality for $m_R$-primary ideals  do not hold for filtrations.

In this article, we consider multiplicities and mixed multiplicities of $R$ by $m_R$-primary divisorial filtrations. We show that some important theorems on equalities of multiplicities and mixed multiplicities of $m_R$-primary ideals, which are not true in general for filtrations, are true for divisorial filtrations. 
We prove that a theorem of Rees showing that if there is an  inclusion of $m_R$-primary ideals $I\subset I'$ with the same multiplicity then $I$ and $I'$ have the same integral closure 
also holds  for divisorial filtrations. This theorem does not hold for arbitrary filtrations. The classical Minkowski inequalities for $m_R$-primary ideals $I_1$ and $I_2$ hold quite generally for filtrations. If $R$ has dimension two and there is equality in the Minkowski inequalities, then Teissier and Rees and Sharp have shown that there are powers $I_1^a$ and $I_2^b$ which have the same integral closure. This theorem does not hold for arbitrary filtrations. The Teissier Rees Sharp theorem has been extended  by Katz to $m_R$-primary ideals in arbitrary dimension. We show that  the Teissier Rees Sharp theorem does hold for divisorial filtrations 
in an excellent domain of dimension two.

We also show that the mixed multiplicities of divisorial filtrations are anti-positive intersection products on a suitable normal scheme $X$ birationally dominating $R$, when $R$ is an algebraic local domain.
\end{abstract}


\maketitle
\section{Introduction}

The study of mixed multiplicities of $m_R$-primary ideals in a Noetherian local ring $R$ with maximal ideal $m_R$  was initiated by Bhattacharya \cite{Bh}, Rees  \cite{R} and Teissier  and Risler \cite{T1}. 
In \cite{CSS}  the notion of mixed multiplicities is extended to arbitrary,  not necessarily Noetherian, filtrations of $R$ by $m_R$-primary ideals. It is shown in \cite{CSS} that many basic theorems for mixed multiplicities of $m_R$-primary ideals are true for filtrations. 

The development of the subject of mixed multiplicities and its connection to Teissier's work on equisingularity \cite{T1} can be found in \cite{GGV}.   A  survey of the theory of  mixed multiplicities of  ideals  can be found in  \cite[Chapter 17]{HS}, including discussion of the results of  the papers \cite{R1} of Rees and \cite{S} of  Swanson, and the theory of Minkowski inequalities of Teissier \cite{T1}, \cite{T2}, Rees and Sharp \cite{RS} and Katz \cite{Ka}.   Later, Katz and Verma \cite{KV}, generalized mixed multiplicities to ideals which are not all $m_R$-primary.  Trung and Verma \cite{TV} computed mixed multiplicities of monomial ideals from mixed volumes of suitable polytopes.  
%Mixed multiplicities are also used by Huh in the analysis of the coefficients of the chromatic polynomial of graph theory in \cite{H}.

We will be concerned with  multiplicities and mixed multiplicities of  (not necessarily Noetherian) filtrations, which are defined as follows. 

\begin{Definition}
A filtration $\mathcal I=\{I_n\}_{n\in\NN}$ of  a ring $R$ is a descending chain
$$
R=I_0\supset I_1\supset I_2\supset \cdots
$$
of ideals such that $I_iI_j\subset I_{i+j}$ for all $i,j\in \NN$.  A filtration $\mathcal I=\{I_n\}$ of  a local ring $R$ by $m_R$-primary ideals is a filtration $\mathcal I=\{I_n\}_{n\in\NN}$ of $R$ such that   $I_n$ is $m_R$-primary for $n\ge 1$.
A filtration $\mathcal I=\{I_n\}_{n\in\NN}$ of  a ring $R$ is said to be Noetherian if $\bigoplus_{n\ge 0}I_n$ is a finitely generated $R$-algebra.
 \end{Definition}

The following theorem is the key result needed to define the multiplicity of a filtration of $R$ by $m_R$-primary ideals.  Let $\ell_R(M)$ denote the length of an $R$-module $M$.

\begin{Theorem} \label{TheoremI20} (\cite[Theorem 1.1]{C2} and  \cite[Theorem 4.2]{C3}) Suppose that $R$ is a Noetherian local ring of dimension $d$, and  $N(\hat R)$ is the nilradical of the $m_R$-adic completion $\hat R$ of $R$.  Then   the limit 
\begin{equation}\label{I5}
\lim_{n\rightarrow\infty}\frac{\ell_R(R/I_n)}{n^d}
\end{equation}
exists for any filtration  $\mathcal I=\{I_n\}$ of $R$ by $m_R$-primary ideals, if and only if $\dim N(\hat R)<d$.
\end{Theorem}

When the ring $R$ is a domain and is essentially of finite type over an algebraically closed field $k$ with $R/m_R=k$, Lazarsfeld and Musta\c{t}\u{a} \cite{LM} showed that
the limit exists for all filtrations  of $R$ by $m_R$-primary ideals.  Cutkosky \cite{C3} proved it in the complete generality  stated above in Theorem \ref{TheoremI20}.

As can be seen from this theorem,  one must impose the condition that 
the dimension of the nilradical of the completion $\hat R$ of $R$ is less than the dimension of $R$. The nilradical $N(R)$ of a $d$-dimensional ring $R$ is 
$$
N(R)=\{x\in R\mid x^n=0 \mbox{ for some positive integer $n$}\}.
$$
We have that $\dim N(R)=d$ if and only if there exists a minimal prime $P$ of $R$ such that $\dim R/P =d$ and $R_P$ is not reduced. In particular,  the condition $\dim N(\hat R)<d$ holds if $R$ is analytically unramified; that is, $\hat R$ is reduced. 
We define the multiplicity of $R$ with respect to the filtration $\mathcal I=\{I_n\}$ to be 
$$
e_R(\mathcal I;R)=
\lim_{n\rightarrow \infty}\frac{\ell_R(R/I_n)}{n^d/d!}.
$$

The multiplicity of a ring with respect to a non Noetherian filtration can be an irrational number. 
A simple example on a regular local ring is given in \cite{CSS}.

Mixed multiplicities of filtrations are defined in \cite{CSS}. 
 Let $M$ be a finitely generated $R$-module where $R$ is a $d$-dimensional Noetherian local ring with $\dim N(\hat R)<d$. Let $\mathcal I(1)=\{I(1)_n\},\ldots, \mathcal I(r)=\{I(r)_n\}$ be filtrations of $R$ by $m_R$-primary ideals. 
 In  \cite[Theorem 6.1]{CSS} and  \cite[Theorem 6.6]{CSS}, it is shown that the function
\begin{equation}\label{M2}
P(n_1,\ldots,n_r)=\lim_{m\rightarrow \infty}\frac{\ell_R(M/I(1)_{mn_1}\cdots I(r)_{mn_r}M)}{m^d}
\end{equation}
is equal to a homogeneous polynomial $G(n_1,\ldots,n_r)$ of total degree $d$ with real coefficients for all  $n_1,\ldots,n_r\in\NN$.  

We  define the mixed multiplicities of $M$ from the coefficients of $G$, generalizing the definition of mixed multiplicities for $m_R$-primary ideals. Specifically,   
 we write 
\begin{equation}\label{eqV6}
G(n_1,\ldots,n_r)=\sum_{d_1+\cdots +d_r=d}\frac{1}{d_1!\cdots d_r!}e_R(\mathcal I(1)^{[d_1]},\ldots, \mathcal I(r)^{[d_r]};M)n_1^{d_1}\cdots n_r^{d_r}.
\end{equation}
We say that $e_R(\mathcal I(1)^{[d_1]},\ldots,\mathcal I(r)^{[d_r]};M)$ is the mixed multiplicity of $M$ of type $(d_1,\ldots,d_r)$ with respect to the filtrations $\mathcal I(1),\ldots,\mathcal I(r)$.
Here we are using the notation 
\begin{equation}\label{eqI6}
e_R(\mathcal I(1)^{[d_1]},\ldots, \mathcal I(r)^{[d_r]};M)
\end{equation}
  to be consistent with the classical notation for mixed multiplicities of $M$ with respect to $m_R$-primary ideals from \cite{T1}. The mixed multiplicity of $M$ of type $(d_1,\ldots,d_r)$ with respect to $m_R$-primary ideals $I_1,\ldots,I_r$, denoted by $e_R(I_1^{[d_1]},\ldots,I_r^{[d_r]};M)$ (\cite{T1}, \cite[Definition 17.4.3]{HS}) is equal to the mixed multiplicity $e_R(\mathcal I(1)^{[d_1]},\ldots,\mathcal I(r)^{[d_r]};M)$, where the Noetherian $I$-adic filtrations $\mathcal I(1),\ldots,\mathcal I(r)$ are defined by $\mathcal I(1)=\{I_1^i\}_{i\in \NN}, \ldots,\mathcal I(r)=\{I_r^i\}_{i\in \NN}$.

We have that 
\begin{equation}\label{eqX31}
e_R(\mathcal I;M)=e_R(\mathcal I^{[d]};M)
\end{equation}
 if $r=1$, and $\mathcal I=\{I_i\}$ is a filtration of  $R$ by $m_R$-primary ideals. We have that
$$
e_R(\mathcal I;M)=\lim_{m\rightarrow \infty}d!\frac{\ell_R(M/I_mM)}{m^d}.
$$
%%%%%%%%%%%%%%

The multiplicities and mixed multiplicities of $m_R$-primary ideals  are always positive (\cite{T1} or \cite[Corollary 17.4.7]{HS}). The multiplicities and mixed multiplicities of filtrations are always nonnegative, as is established in \cite[Proposition 1.3]{CSV}, but can be zero. If $R$ is analytically irreducible, then all mixed multiplicities are positive if and only if the multiplicities $e_R(\mathcal I(j);R)$ are positive for $1\le j\le r$. This is established in \cite[Theorem 1.4]{CSV}.

Suppose that $R$ is a $d$-dimensional excellent local domain, with quotient field $K$. A valuation $\nu$ of $K$ is called an $m_R$-valuation if $\nu$ dominates $R$ ($R\subset V_{\nu}$ and $m_{\nu}\cap R=m_R$ where $V_{\nu}$ is the valuation ring of $\nu$ with maximal ideal $m_{\nu}$) and ${\rm trdeg}_{R/m_R}V_{\nu}/m_{\nu}=d-1$.

Suppose that $I$ is an ideal in $R$. Let $X$ be the normalization of the blowup of $I$, with projective birational morphism $\phi:X\rightarrow \mbox{Spec}(R)$. Let $E_1,\ldots,E_t$ be the irreducible components of $\phi^{-1}(V(I))$ (which necessarily have dimension $d-1$). The Rees valuations of $I$ are the discrete valuations $\nu_i$ for $1\le i\le t$ with valuation rings $V_{\nu_i}=\mathcal O_{X,E_i}$. If $R$ is normal, then $X$ is equal to the blowup of the integral closure $\overline{I^s}$ of an appropriate power $I^s$ of $I$.

Every Rees valuation $\nu$ which dominates $R$  is an $m_R$-valuation and every $m_R$-valuation is a Rees valuation of an $m_R$-primary ideal by \cite[Statement  (G)]{R3}.

Associated to an $m_R$-valuation $\nu$ are valuation ideals
\begin{equation}\label{eqX2}
I(\nu)_n=I_R(\nu)=\{f\in R\mid \nu(f)\ge n\}
\end{equation}
for $n\in \NN$.
In general, the filtration $\mathcal I(\nu)=\{I(\nu)_n\}$ is not Noetherian. 
In  a two-dimensional normal local ring $R$, the condition that the  filtration of valuation ideals of $R$ is Noetherian for all  $m_R$-valuations dominating $R$ is the condition (N) of Muhly and Sakuma \cite{MS}. It is proven in \cite{C6} that a complete normal local ring of dimension two satisfies condition (N) if and only if its divisor class group is a torsion group. 
An example is given in \cite{CGP} of an $m_R$-valuation of a 3-dimensional regular local ring $R$ which is not Noetherian. 

\begin{Definition}\label{DefDF} Suppose that $R$ is an excellent local domain. We
say that a filtration $\mathcal I$ of $R$ by $m_R$-primary ideals is a divisorial filtration if there exists a projective birational morphism $\phi:X\rightarrow \mbox{Spec}(R)$ such that $X$ is the normalization of the blowup of an $m_R$-primary ideal and there exists a nonzero effective Cartier divisor $D$ on $X$ with exceptional support for $\phi$ such that 
$\mathcal I=\{I(mD)\}_{m\in\NN}$ where
\begin{equation}\label{eqX22}
I(mD)=I_R(mD)=\Gamma(X,\mathcal O_X(-mD))\cap R .
\end{equation}
\end{Definition}
If $R$ is normal, then $I(mD)=\Gamma(X,\mathcal O_X(-mD))$.
If $D=\sum_{i=1}^ta_iE_i$ where the $a_i\in \NN$ and the $E_i$ are prime exceptional divisors of $\phi$, with associated $m_R$-valuations $\nu_i$, then 
$$
I(mD)=I(\nu_1)_{a_1m}\cap\cdots\cap I(\nu_t)_{a_tm}.
$$

Suppose  that $\mathcal I(1),\ldots,\mathcal I(r)$ are divisorial filtrations of  an excellent local domain $R$.    We then have associated mixed multiplicities
\begin{equation}\label{eqX20}
e_R(\mathcal I(1)^{[d_1]},\ldots, \mathcal I(r)^{[d_r]};R)
\end{equation}
for $d_1,\ldots,d_r\in \NN$ with $d_1+\cdots+d_r=d$.

If $R$ is analytically irreducible, then all mixed multiplicities  (\ref{eqX20}) are  positive by Proposition \ref{PropPos}.

We show in (\ref{eq15}) and (\ref{eq14}) of Section \ref{Sec5} that if $R$ has dimension two, then the mixed multiplicities (\ref{eqX20}) are positive rational numbers. 
In Example 6 of \cite{CS}, an example is given of an $m_R$-valuation $\nu$ dominating a normal  excellent local domain of dimension three such that $e_R(\mathcal I(\nu);R)$ is an irrational number. Thus the mixed multiplicities (\ref{eqX20}) can be irrational if $d\ge 3$.

%%%%%%%%%%%%%%%%%%%%%%%%

The following theorem in \cite{CSS} generalizes  \cite[Proposition 11.2.1]{HS} for $m_R$-primary ideals to filtrations of $R$ by $m_R$-primary ideals.

\begin{Theorem}\label{Theorem13}(\cite[Theorem 6.9]{CSS}) Suppose that $R$ is a Noetherian $d$-dimensional local ring such that 
$$
\dim N(\hat R)<d
$$
 and $M$ is a finitely generated $R$-module. Suppose that $\mathcal I'=\{I'_i\}$ and $\mathcal I=\{I_i\}$ are filtrations of $R$ by $m_R$-primary ideals. Suppose that $\mathcal I'\subset \mathcal I$ ($
I'_i\subset I_i$ for all $i$) and the ring $\bigoplus_{n\ge 0}I_n$ is integral over $\bigoplus_{n\ge 0} I'_n$. Then 
$$
e_R(\mathcal I;M)=e_R(\mathcal I';M).
$$
\end{Theorem}
We give a proof of Theorem \ref{Theorem13} in the Appendix. 

Rees has shown in \cite{R} that if $R$ is a formally equidimensional Noetherian local ring and $I\subset I'$ are $m_R$-primary ideals such that $e_R(I;R)=e_R(I';R)$, then $\bigoplus_{n\ge 0}(I')^n$ is integral over $\bigoplus_{n\ge 0}I^n$ ($I$ and $I'$ have the same integral closure). An exposition of this converse to the above cited \cite[Proposition 11.2.1]{HS} is given in   \cite[Proposition 11.3.1]{HS}, in the section entitled ``Rees's Theorem''. Rees's theorem is not true in general for filtrations of $m_R$-primary ideals (a simple example in a regular local ring is given in \cite{CSS}) but it is true for divisorial filtrations. 
In Theorem \ref{TheoremX1}, we show that Rees's theorem (the converse of Theorem \ref{Theorem13}) is true for divisorial filtrations of an excellent local domain.

An analogue of the Rees theorem for projective varieties is proven in Theorem \ref{TheoremGRT}.

We prove in  \cite[Theorem 6.3]{CSS} that the Minkowski inequalities hold for filtrations of $m_R$-primary ideals.

\begin{Theorem}(Minkowski Inequalities for filtrations)(\cite[Theorem 6.3]{CSS})\label{TheoremMI}  Suppose that $R$ is a Noetherian $d$-dimensional  local ring with $\dim N(\hat R)<d$, $M$ is a finitely generated $R$-module and $\mathcal I(1)=\{I(1)_j\}$ and $\mathcal I(2)=\{I(2)_j\}$ are filtrations of $R$ by $m_R$-primary ideals. Then 
\begin{enumerate}
\item[1)] $e_R(\mathcal I(1)^{[i]},\mathcal I(2)^{[d-i]};M)^2\le e_R(\mathcal I(1)^{[i+1]},\mathcal I(2)^{[d-i-1]};M)e_R(\mathcal I(1)^{[i-1]},\mathcal I(2)^{[d-i+1]};M)$ 

for $1\le i\le d-1$.
\item[2)]  For $0\le i\le d$, 
$$
e_R(\mathcal I(1)^{[i]},\mathcal I(2)^{[d-i]};M)e_R(\mathcal I(1)^{[d-i]},\mathcal I(2)^{[i]};M)\le e_R(\mathcal I(1);M)e_R(\mathcal I(2);M),
$$
\item[3)] For $0\le i\le d$, $e_R(\mathcal I(1)^{[d-i]},\mathcal I(2)^{[i]};M)^d\le e_R(\mathcal I(1);M)^{d-i}e_R(\mathcal I(2);M)^i$ and
\item[4)]  $e_R(\mathcal I(1)\mathcal I(2));M)^{\frac{1}{d}}\le e_R(\mathcal I(1);M)^{\frac{1}{d}}+e_R(\mathcal I(2);M)^{\frac{1}{d}}$, 

where $\mathcal I(1)\mathcal I(2)=\{I(1)_jI(2)_j\}$.
\end{enumerate}
\end{Theorem}

The Minkowski inequalities were formulated and proven for  $m_R$-primary ideals by Teissier \cite{T1}, \cite{T2} and proven in full generality, for Noetherian local rings,  by Rees and Sharp \cite{RS}.
 The fourth inequality  4)  was proven for filtrations  of $R$ by $m_R$-primary ideals in a regular local ring with algebraically closed residue field by Musta\c{t}\u{a} (\cite[Corollary 1.9]{Mus}) and more recently by Kaveh and Khovanskii (\cite[Corollary 7.14]{KK1}). The inequality 4) was proven with our assumption that $\dim N(\hat R)<d$ in \cite[Theorem 3.1]{C3}.
Inequalities 2) - 4) can be deduced directly from inequality 1), as explained in \cite{T1}, \cite{T2}, \cite{RS} and \cite[Corollary 17.7.3]{HS}.

Teissier \cite{T3} (for Cohen Macaulay normal two-dimensional complex analytic $R$), Rees and Sharp \cite{RS} (in dimension 2) and Katz \cite{Ka} (in complete generality) have proven that if $R$ is a $d$-dimensional formally equidimensional Noetherian local ring and $I(1)$, $I(2)$ are $m_R$-primary ideals such that the Minkowski equality
$$
e_R((I(1) I(2));R)^{\frac{1}{d}}= e_R( I(1);R)^{\frac{1}{d}}+e_R(I(2);R)^{\frac{1}{d}}
$$
holds,
 then there exist positive integers $r$ and $s$ such that the integral closures 
 $\overline {I(1)^r}$ and $\overline{I(2)^s}$ of the ideals $I(1)^r$ and $S(2)^s$ are equal, which is equivalent to the statement that the $R$-algebras $\bigoplus_{n\ge 0}I(1)^n$ and $\bigoplus_{n\ge 0}I(2)^n$ have the same integral closure.

 The Teissier Rees Sharp Katz theorem is not true for filtrations, even in a regular local ring, as is shown in a simple example in \cite{CSS}.
 
 In Theorem \ref{PropR11}, we show that the Teissier Rees Sharp theorem is true for divisorial filtrations of an excellent two-dimensional local domain.

  In Section \ref{SecAPM1}, we interpret the mixed multiplicities of divisorial  filtrations $\mathcal I(1),\ldots,\mathcal I(r)$ as intersection multiplicities. We assume that $R$ is an algebraic local domain; that is, a  domain that is essentially of finite type over an arbitrary  field $k$  (a localization of a finitely generated $k$-algebra), and that $\phi:X\rightarrow \mbox{Spec}(R)$ is the normalization of the blowup of an $m_R$-primary ideal. We define in Section \ref{SecAPM3} anti-positive intersection products $\langle F_1,\ldots,F_d\rangle$ of anti-effective Cartier divisors $F_1,\ldots,F_d$ on $X$ with exceptional support for $\phi$, generalizing the positive intersection product of Cartier divisors defined on projective varieties in \cite{BFJ} over an algebraically closed field of characteristic zero and in \cite{C4} over an arbitrary field. 
 
 Suppose that $D(1),\ldots,D(r)$ are Cartier divisors on $X$ with exceptional support. Let $\mathcal I(j)=\{I(nD(j))\}$ for $1\le i\le r$ be divisorial filtrations of $R$, where the $m_R$-primary ideals $I(nD(j))$ are defined by (\ref{eqX22}).

 In Theorem \ref{TheoremA}, we show that, when $R$ is normal, the mixed multiplicities 
 $$
 e_R(\mathcal(1)^{[d_1]},\ldots,\mathcal I(r)^{[d_r]};R)=-\langle(-D(1))^{d_1},\ldots,(-D(r))^{d_r}\rangle
 $$
 are the negatives of the corresponding anti-positive intersection multiplicities
 for all 
 $$
 d_1,\ldots,d_r\in \NN
 $$
  such that $d_1+\cdots+d_r=d$. A related formula is given in Theorem \ref{Theorem10} if $R$ is not normal. 
 
 When $R$ has dimension 2, the anti-positive intersection product 
 $$
 \langle(-D(1))^{d_1},(-D(2))^{d_2}\rangle=(\Delta_1^{d_1}\cdot\Delta_2^{d_2})
 $$
 is the ordinary intersection product of the anti-nef parts $\Delta_1$, $\Delta_2$ of the respective Zariski decompositions of $D_1$ and $D_2$. 
 
 In Section \ref{Sec5}, we develop the theory of mixed multiplicities of divisorial filtrations in a two-dimensional excellent local domain using the theory of Zariski decomposition. We give a proof of Theorem \ref{TheoremX1} in dimension 2 using this method in Proposition \ref{PropR10} and use this method to prove Proposition \ref{PropR11} on the Minkowski equality.

 We use the method of volumes of convex bodies associated to appropriate semigroups
 introduced in \cite{Ok}, \cite{LM} and \cite{KK}.

We will denote the nonnegative integers by $\NN$ and the positive integers by $\ZZ_+$.  We will denote the set of nonnegative rational numbers  by $\QQ_{\ge 0}$  and the positive rational numbers by $\QQ_+$. 
We will denote the set of nonnegative real numbers by $\RR_{\ge0}$. For a real number $x$, $\lceil x\rceil$ will denote the smallest integer which is $\ge x$ and $\lfloor x\rfloor$ will denote the largest integer which is $\le x$.

The maximal ideal of a local ring $R$ will be denoted by $m_R$. The quotient field of a domain $R$ will be denoted by ${\rm QF}(R)$. We will denote the length of an $R$-module $M$ by $\ell_R(M)$.

\section{First Properties of Mixed multiplicities of divisorial filtrations} \label{Sec6}
\label{Secd-dim}

In this section we prove some basic facts about mixed multiplicities of valuation ideals 
an divisorial filtrations which will be useful. 

\begin{Proposition}\label{PropPos}  Suppose that $R$ is an excellent, analytically irreducible $d$-dimensional local domain and $\nu_1,\ldots,\nu_t$ are $m_R$-valuations of $R$. 
\begin{enumerate}
\item[1)] Suppose that $a_1,\ldots,a_t\in \NN$ are not all zero. Let 
$I_n=I(\nu_1)_{na_1}\cap \cdots\cap I(\nu_t)_{na_t}$ and $\mathcal I=\{I_n\}$. Then 
$$
e_R(\mathcal I;R)>0.
$$
\item[2)] Suppose that $r\in \ZZ_+$ and $a_i(j)\in \NN$ for $1\le i\le t$ and $1\le j\le r$ and for each $j$, not all $a_i(j)$ are zero. Let $I(j)_n=I(\nu_1)_{na_1(j)}\cap \cdots \cap I(\nu_t)_{na_t(j)}$ for $1\le j\le r$ and $\mathcal I(j)=\{I(j)_n\}$ for $1\le j\le r$. Then
$$
e_R(\mathcal I(1)^{[d_1]},\ldots,\mathcal I(r)^{[d_r]};R)>0
$$
for all $d_1,\ldots,d_r\in \NN$ with $d_1+\cdots+d_r=d$.
\end{enumerate}
\end{Proposition}

\begin{proof} 
We first prove 1). There exists an $m_R$-primary ideal $J$ such that $\nu_1,\ldots,\nu_t$ are Rees valuations of J. Without loss of generality, we can assume that $\nu_1,\ldots,\nu_t$ are the entirety of the Rees valuations for $J$.
 By Rees's Izumi theorem \cite{R3}, the topologies of the $\nu_i$ are linearly equivalent. Let $\overline \nu_{J}$ be the reduced order. By the Rees valuation theorem (recalled in \cite{R3}), 
$$
\overline \nu_{J}(x)=\min_i\left\{\frac{\nu_i(x)}{\nu_i(J)}\right\}
$$
for $x\in R$, so the topology induced by $\overline\nu_{J}$ is linearly equivalent to the topology induced by the $\nu_i$. We have that $\overline\nu_{J}$ is linearly equivalent to the $J$-topology by \cite{R2} since $R$ is analytically unramified.  

Thus there exists $\alpha\in \ZZ_+$ such that 
\begin{equation}\label{eqX5}
I(\nu_i)_{\alpha n}\subset J^n\subset m_R^n \mbox{ for all $n\in \ZZ_+$.}
\end{equation}
Let $a=\max\{a_1,\ldots,a_t\}$. Then $I_{a\alpha n}\subset m_R^n$ for all $n$.
So $\ell_R(R/m_R^n)\le \ell_R(R/I_{n\alpha a})$ for all $n$ and so
$$
e_R(\mathcal I;R)\ge \frac{1}{(a\alpha)^d}e_R(m_R;R)>0.
$$
We now prove 2). Statement 1) implies that $e_R(\mathcal I(j);R)>0$ for $1\le j\le r$. Thus all mixed multiplicities are positive by \cite[Theorem 1.4]{CSV}.

\end{proof}

Suppose that $R$ is an excellent $d$-dimensional local domain. Let $S$ be the normalization of $R$, which is a finitely generated $R$-module,  and let $m_1,\ldots,m_t$ be the maximal ideals of $S$. Let $\phi:X\rightarrow \mbox{Spec}(R)$ be a birational projective morphism such that $X$ is the normalization of   the blowup of an $m_R$-primary ideal. Since $X$ is  normal, $\phi$ factors through $\mbox{Spec}(S)$. Let $\phi_i:X_i\rightarrow \mbox{Spec}(S_{m_i})$ be the induced  projective morphisms where $X_i=X\times_{\mbox{Spec}(S)}\mbox{Spec}(S_{m_i})$. For $1\le i\le t$, let $\{E_{i,j}\}$ be the irreducible exceptional divisors in $\phi_i^{-1}(m_i)$.

Suppose that $D$ is an effective exceptional Weil divisor on $X$. Write $D=\sum_{i,j} a_{i,j}E_{i,j}$ with  $a_{ij}\in\NN$. Define $D_i=\sum_ja_{i,j}E_{i,j}$ for $1\le i\le t$.  
 The reflexive coherent sheaf $\mathcal O_X(-D)$ of $\mathcal O_X$-modules  is defined by $\mathcal O_X(-D)=i_*\mathcal O_U(-D|U)$ where $U$ is the open subset of regular points of $X$ and $i:U\rightarrow X$ is the inclusion. We have that  $\dim (X\setminus U)\le d-2$ since $X$ is normal. The basic properties of this sheaf are developed for instance in \cite[Section 13.2]{C5}. 
We have that $S\subset \mathcal O_{X,p}$ for all $p\in X$, since $\mathcal O_{X,p}$ is  normal. Now $\Gamma(X,\mathcal O_X)$ is a domain with the same quotient field as $R$, and is a finitely generated $R$-module since $\phi$ is proper. Thus $\Gamma(X,\mathcal O_X)=\Gamma(X,\mathcal O_X(0))=S$.

Let
\begin{equation}\label{eqX30}
J(D)=\Gamma(X,\mathcal O_X(-D)), J(D_i)=\Gamma(X_i,\mathcal O_{X_i}(-D_i)),  I(D)=J(D)\cap R, I(D_i)=J(D_i)\cap R.
\end{equation}
 We have that
\begin{equation}\label{eqR6}
S/J(D)\cong \bigoplus_{i=1}^t S_{m_i}/\Gamma(X_i,\mathcal O_{X_i}(-D_i))\cong \bigoplus_{i=1}^t S_{m_i}/J(D_i)
\end{equation}
and so
\begin{equation}\label{eqR15}
\ell_R(S/J(D))=\sum_{i=1}^t\ell_R(S_{m_i}/J(D_i))=\sum_{i=1}^t[S/m_i:R/m_R]\ell_{S_{m_i}}(S_{m_i}/J(D_i)).
\end{equation}
We have that $[S/m_i:R/m_R]<\infty$ for all $i$ since $S$ is a finitely generated $R$-module. 

Let $D(1),\ldots,D(r)$ be effective Weil divisors on $X$ with exceptional support in $\phi^{-1}(m_R)$.

\begin{Lemma}\label{LemmaR1}   For $n_1,\ldots,n_r\in \NN$,
$$
\lim_{n\rightarrow\infty}\frac{\ell_R(R/I(nn_1D(1))\cdots I(nn_rD(r)))}{n^d}
=
\lim_{n\rightarrow\infty}\frac{\ell_R(S/J(nn_1D(1))\cdots J(nn_rD(r)))}{n^d}.
$$

\end{Lemma}

\begin{proof}  Fix $n_1,\ldots,n_r\in \NN$. Let $\mathcal C$ be the conductor of $R$ (which is a nonzero ideal in both $R$ and $S$), and choose $0\ne x\in\mathcal C$. We then have short exact sequences of $S$-modules
$$
0\rightarrow A_{n}\rightarrow S/J(nn_1D(1))\cdots J(nn_rD(r))\stackrel{x^r}{\rightarrow} 
S/J(nn_1D(1))\cdots J(nn_rD(r))\rightarrow C_{n}\rightarrow 0
$$
where $A_{n}$ and $C_{ n}$ are the respective kernels and cokernels of multiplication of 
$$
S/J(nn_1D(1))\cdots J(nn_rD(r))
$$
 by $x^r$. We have that
$$
C_{n}\cong S/(x^rS+J(nn_1D(1))\cdots J(nn_rD(r)))\cong (S/x^rS)/(J(nn_1D(1))\cdots J(nn_rD(r))(S/x^rS)).
$$
Thus
$\lim_{n\rightarrow \infty}\frac{\ell_S(C_n)}{n^d}=0$ since $\dim S/x^rS=d-1$.
 Now
$$
S/J(nn_1D(1))\cdots J(nn_rD(r))\cong \bigoplus_{j=1}^tS_{m_j}/J(nn_1D(1)_j)\cdots J(nn_rD(r)_j)).
$$
By Theorem \ref{TheoremI20}, the limit
$$
\lim_{n\rightarrow\infty} \frac{\ell_S(S/J(nn_1D(1))\cdots J(nn_rD(r)))}{n^d}
=\sum_{j=1}^t\lim_{n\rightarrow\infty}
\frac{\ell_{S_{m_j}}(S_{m_j}/J(nn_1D(1)_j)\cdots J(nn_rD(r)_j))}{n^d}
$$
exists and so 
$\lim_{n\rightarrow\infty} \frac{\ell_S(A_{ n})}{n^d}=0$.
Let $F_n$ and $B_{ n}$ be the respective kernels and cokernels of the homomorphisms of $R$-modules
$$
S/J(nn_1D(1))\cdots J(nn_rD(r))\stackrel{x^r}{\rightarrow} R/I(nn_1D(1))\cdots I(nn_rD(r))).
$$
Then we have short exact sequences of $R$-modules
$$
0\rightarrow F_{ n}\rightarrow S/J(nn_1D(1))\cdots J(nn_rD(r))\stackrel{x^r}{\rightarrow} R/I(nn_1D(1))\cdots I(nn_rD(r)))\rightarrow B_{ n}\rightarrow 0.
$$
We have natural surjections of $R$-modules
$$
(R/x^rR)/I(nn_1D(1))\cdots I(nn_rD(r))(R/x^rR)\cong    R/(x^rR+I(nn_1D(1))\cdots I(nn_rD(r)))\rightarrow B_{ n}.
$$
Now $\dim R/x^rR=d-1$ so 
$$
\lim_{n\rightarrow \infty}\frac{\ell_R((R/x^rR)/I(nn_1D(1))\cdots I(nn_rD(r))(R/x^rR))}{n^d}=0,
$$
and so 
$$
\lim_{n\rightarrow\infty}\frac{\ell_R(B_n)}{n^d}=0.
$$

Since the support of the $S$-module $A_{ n}$ is contained in the set of maximal ideals $\{m_1,\ldots,m_t\}$, we have that
$A_{ n}\cong \bigoplus_{j=1}^t(A_{n})_{m_j}$ and $\ell_S(A_{ n})=\sum_{j=1}^t\ell_{S_{m_j}}((A_{ n})_{m_j})$. Thus 
$$
\begin{array}{lll}
\ell_R(A_{ n})&=& \sum_{j=1}^t [S/m_j:R/m_R]\ell_{S_{m_j}}((A_{ n})_{m_j})\\
&\le&\mu \ell_S(A_{ n})
\end{array}
$$
where $\mu=\max_j\{[S/m_j:R/m_R]\}$.
We then have that
$$
lim_{n\rightarrow\infty}\frac{\ell_R(A_{ n})}{n^d}\le\mu \lim_{n\rightarrow\infty}\frac{\ell_S(A_{ n})}{n^d}=0.
$$
There are natural inclusions $F_n\subset A_n$ for all $n$, so
$$
\lim_{n\rightarrow \infty}\frac{\ell_R(F_n)}{n^d}=0
$$
and thus
$$
\lim_{n\rightarrow\infty}\frac{\ell_R(R/I(nn_1D(1))\cdots I(nn_rD(r)))}{n^d}
=
\lim_{n\rightarrow\infty}\frac{\ell_R(S/J(nn_1D(1))\cdots J(nn_rD(r)))}{n^d}.
$$
\end{proof}

\section{Rees's theorem for divisorial filtrations}

In this section, suppose that $R$ is a $d$-dimensional normal excellent local ring. Let $\phi:X\rightarrow \mbox{Spec}(R)$ be a birational projective morphism which is the   blowup of an $m_R$-primary ideal such that $X$ is normal.

Let $E_1,\ldots,E_r$ be the prime exceptional divisors of $\phi$ (which all contract to $m_R$), and let $\mu_i$ be the discrete valuation with valuation ring $\mathcal O_{X,E_i}$ for $1\le i\le r$. Let $D$ be a nonzero effective Cartier divisor on $X$ with exceptional support. Let
$$
I(\mu_i)_n=\{f\in R\mid \mu_i(f)\ge n\}.
$$

For $1\le i\le r$ and 
$m\in \NN$, define
$$
\tau_{E_i,m}(D)=\min\{\mu_i(f)\mid f\in \Gamma(X,\mathcal O_X(-mD))\}.
$$
Let $\tau_m=\tau_{E_i,m}(D)$. Then since $\tau_{mn}\le n\tau_m$, we have that
\begin{equation}\label{eqAR1}
\frac{\tau_{mn}}{mn}\le\min\{\frac{\tau_m}{m},\frac{\tau_n}{n}\}.
\end{equation}

Now define
$$
\gamma_E(D)=\inf_m\frac{\tau_m}{m}.
$$

Expand $D=\sum_{i=1}^r a_iE_i$ with  $a_i\in \NN$. 
We have that 
$$
\Gamma(X,\mathcal O_X(-mD))=\{f\in R\mid \mu_i(f)\ge ma_i\mbox{ for }1\le i\le r\}.
$$
Thus $\tau_{E_i,m}(D)\ge ma_i$ for all $m\in\NN$, and so 
\begin{equation}\label{eqAR12}
\gamma_{E_i}(D)\ge a_i\mbox{ for all $i$}.
\end{equation}

\begin{Lemma}\label{LemmaAR1}
We have that
$$
\Gamma(X,\mathcal O_X(-mD))
=\Gamma(X,\mathcal O_X(-\lceil \sum_{i=1}^r m\gamma_{E_i}(D)E_i\rceil))
$$
for all $m\in \NN$.
\end{Lemma}

\begin{proof} We have that 
$$
\Gamma(X,\mathcal O_X(-\lceil \sum_{i=1}^rm\gamma_{E_i}(D)E_i\rceil))
\subset \Gamma(X,\mathcal O_X(-mD))
$$
by (\ref{eqAR12}).

Suppose that $f\in \Gamma(X,\mathcal O_X(-mD))$. Then 
$\mu_i(f)\ge\tau_{E_i,m}(D)\ge m\gamma_{E_i}(D)$ for all $i$, so that
$\mu_i(f)\ge \lceil m\gamma_{E_i}(D)\rceil$ for all $i$ since $\mu_i(f)\in \NN$.
\end{proof}

We now define a  valuation which we will use  to compute volumes of Cartier divisors $D$, and  which will allow us to extract some extra information which we need to prove Theorem \ref{TheoremAR1} below.
Suppose that $p\in E_i$ is a  closed point which   is nonsingular on $X$ and $E_i$  and which
is not contained in $E_j$ for $j\ne i$. Let
\begin{equation}\label{eqAR2}
X=Y_0\supset Y_1=E_i\supset \cdots \supset Y_d=\{p\}
\end{equation}
be a flag; that is, the $Y_i$ are subvarieties of $X$ of dimension $d-i$ such that there is a regular system of parameters $a_1,\ldots,a_d$ in $\mathcal O_{X,p}$ such that  $a_1=\cdots=a_i=0$ are local equations of $Y_i$ for $1\le i\le d$. 

The flag determines a valuation $\nu$ on the quotient field $K$ of $R$  as follows.  We have a sequence of natural surjections of regular local rings
\begin{equation}\label{eqGA3} \mathcal O_{X,p}=
\mathcal O_{Y_0,p}\overset{\sigma_1}{\rightarrow}
\mathcal O_{Y_1,p}=\mathcal O_{Y_0,p}/(a_1)\overset{\sigma_2}{\rightarrow}
\cdots \overset{\sigma_{d-1}}{\rightarrow} \mathcal O_{Y_{d-1},p}=\mathcal O_{Y_{d-2},p}/(a_{d-1}).
\end{equation}
Define a rank $d$ discrete valuation $\nu$ on $K$ (an Abhyankar valuation)  by prescribing for $s\in \mathcal O_{X,p}$,
$$
\nu(s)=({\rm ord}_{Y_1}(s),{\rm ord}_{Y_2}(s_1),\cdots,{\rm ord}_{Y_d}(s_{d-1}))\in (\ZZ^d)_{\rm lex}
$$
where 
$$
s_1=\sigma_1\left(\frac{s}{a_1^{{\rm ord}_{Y_1}(s)}}\right),
s_2=\sigma_2\left(\frac{s_1}{a_2^{{\rm ord}_{Y_2}(s_1)}}\right),\ldots,
s_{d-1}=\sigma_{d-1}\left(\frac{s_{d-2}}{a_{d-1}^{{\rm ord}_{Y_{d-1}}(s_{d-2})}}\right)
$$
and $\mbox{ord}_{Y_{j+1}}(s_j)$ is the highest power of $a_{j+1}$  which divides $s_j$ in $\mathcal O_{Y_j,p}$.
We have that
$$
\nu(s)=\left(\mu_i(f),\omega\left(\frac{f}{a_1^{\mu_1(f)}}\right)\right)
$$
where $\omega$ is the rank $d-1$ Abhyankar valuation on the function field $k(E_i)$ of $E_i$ determined by the flag
$$
E_i=Y_1\supset \cdots \supset Y_d=\{p\}.
$$
on the projective $k$-variety $E_i$, where $k=R/m_R$.

Consider the graded linear series $\Gamma(E_i,\mathcal O_X(-nE_i)\otimes_{\mathcal O_X}\mathcal O_{E_i})$ on $E_i$. Let $g=0$ be a local equation of $E_i$ in $\mathcal O_{X,p}$. Then for $n\in \NN$, we have natural commutative diagrams
$$
\begin{array}{ccc}
\Gamma(X,\mathcal O_X(-nE_i))&\rightarrow & \Gamma(E_i,\mathcal O_X(-nE_i)\otimes\mathcal O_{E_i})\\
\downarrow&&\downarrow\\
\mathcal \mathcal O_X(-nE_i)_p&\rightarrow&\mathcal O_X(-nE_i)_p\otimes_{\mathcal O_{X,p}}\mathcal O_{E_i,p}\\
=O_{X,p}g^n&&\cong \mathcal O_{E_i,p}\otimes_{\mathcal O_{X,p}}\mathcal O_{X,p}g^n\end{array}
$$
where we denote the rightmost vertical arrow by $s\mapsto \epsilon_n(s)\otimes g^n$ 
and the bottom horizontal arrow is 
$$
f\mapsto \left[\frac{f}{g^n}\right]\otimes g^n,
$$
where $\left[\frac{f}{g^n}\right]$ is the class of $\frac{f}{g^n}$ in $\mathcal O_{E_i,p}$.

Let $\Xi$ be the semigroup
$$
\Xi = \{(\omega(\epsilon_n(s)),n)\mid n\in \NN\mbox{ and } s\in \Gamma(E_i,\mathcal O_X(-nE_i)\otimes_{\mathcal O_X}\mathcal O_{E_i})\}\subset \ZZ^d,
$$
and let $\Delta(\Xi)$ be the intersection of the closed convex cone generated by $\Xi$ in $\RR^d$ with $\RR^{d-1}\times\{1\}$. By the proof of Theorem 8.1 \cite{C2} or the proof of \cite[Theorem A]{LM}, $\Delta(\Xi)$ is  compact and convex. Let 
$$
\Xi_n=\{(\omega(\epsilon_n(s)),n)\mid s\in \Gamma(E_i,\mathcal O_X(-nE_i)\otimes_{\mathcal O_X}\mathcal O_{E_i})\}.
$$
Suppose that $\delta$ is a positive integer.
Let 
$$
\Gamma(D)=\{(\nu(f),n)\mid f\in I(nD)\mbox{ and }\mu_1(f)\le n\delta\}\subset \NN^{d+1}.
$$
Let $\Delta(D)$ be the intersection of the closed convex cone generated by $\Gamma(D)$ in $\RR^{d+1}$ with $\RR^d\times\{1\}$. 

We have that 
$$
\Gamma(D)_m :=\{(\nu(f),m)\mid f\in I(mD)\}
\subset \cup_{0\le i\le m\delta}\left(\{i\}\times \Xi_i\right)\times\{m\}.
$$
For $t\in \RR_+$, let $t\Delta(\Xi)=\{t\sigma\mid \sigma \in \Delta(\Xi)\}$.  For $(i,\sigma,m)\in \Gamma(D)$, we have that
$$
\left(\frac{i}{m},\frac{\sigma}{m}\right)\in \cup_{0\le i\le \delta m}\left[\{\frac{i}{m}\}\times \frac{i}{m}\Delta(\Xi)\right]\subset \cup_{t\in [0,\delta]}\{t\}\times t\Delta(\Xi).
$$
The continuous map $[0,\delta]\times\Delta(\xi)\rightarrow \RR^d$ defined by $(t,x)\mapsto (t,tx)$ has image $\cup_{t\in [0,\delta]}\{t\}\times t\Delta(\Xi)$ which is compact since $\Delta(\Xi)$ is.
Thus the closed convex set $\Delta(D)$ is compact  and so $\Gamma(D)$  satisfies condition (5) of \cite[Theorem 3.2]{C2}.

%if $f\in m_R$ then the divisor of $f$ on $X$ satisfies $(f)\ge \sum_{i=1}^rE_i$, so $\mu_i(f)\ge 1$, and thus $\Gamma(D)\cap (\NN^d\times\{0\})=\{0\}$, which is the condition (2.3) of page 16 of \cite{LM}.

Now we verify that condition (6) of \cite[Theorem  3.2]{C2} 
 is satisfied; that is, $\Gamma(D)$ generates $\ZZ^{d+1}$ as a group. let $G(\Gamma(D))$ be the subgroup of $\ZZ^{d+1}$ generated by $\Gamma(D)$.
We have that the value group of $\nu$ is $\ZZ^d$, and $e_i=\nu(a_i)$ for $1\le i\le d$ is the natural basis of $\ZZ^d$.  Write $a_i=\frac{f_i}{g_i}$ with $f_i,g_i\in R$ for $1\le i\le d$. There exists $0\ne h\in I(D)$. Thus $hf_i,hg_i\in I(D)$. There exists $c\in \ZZ_+$ such that $hf_i,hg_i\not\in I(\mu_1)_c$ for $1\le i\le d$. Possibly increasing $\delta$ in the definition of $\Gamma(D)$, we then have 
$(\nu(hf_i),1),(\nu(hg_i),1)\in \Gamma(D)$ for $1\le i\le d$. Thus 
$(\nu(hf_i)-\nu(hg_i),0)=(e_i,0)\in G(\Gamma(D))$ for $1\le i\le d$. Since $(\nu(hf_i),1)\in \Gamma(D)$, we then have that $(0,1)\in G(\Gamma(D))$.
Thus we have that
$$
\lim_{n\rightarrow \infty}\frac{\# \Gamma(D)_n}{n^d}={\rm Vol}(\Delta(D))
$$
by \cite[Theorem 3.2]{C2} or \cite[Proposition 2.1]{LM}.

By Rees's Izumi theorem \cite{R3}, we have that there exists $\lambda\in\ZZ_+$ such that if $f\in R$ and $\mu_i(f)\ge n\lambda$, then $\mu_j(f)\ge n$ for $1\le j\le r$. Thus $I(\mu_i)_{n\lambda}\subset I(\mu_j)_n$ for all $n\in \NN$, so that
$$
I(\mu_i)_{na\lambda}\subset I(\mu_i)_{na_1}\cap \cdots \cap I(\mu_r)_{na_r}=\Gamma(X,\mathcal O_X(-nD))
$$
where $a=\max\{a_1,\ldots,a_r\}$. 

Take $\delta$ to be greater than or equal to  $a\lambda$ in the definition of $\Gamma(D)$. 
Let 
$$
\mu=[\mathcal O_{X,p}/m_p:R/m_R].
$$
 Consider the Newton Okounkov bodies 
$\Delta(0)$ and $\Delta(D)$ constructed from the semigroups 
$\Gamma(0)$ and $\Gamma(D)$ with  this $\delta$. Then, as in \cite[Theorem 5.6]{C3},
\begin{equation}\label{eqAR10}
\lim_{m\rightarrow\infty}\frac{\ell_R(R/I(mD))}{m^d}=\mu({\rm Vol}(\Delta(0))-{\rm Vol}(\Delta(D))).
\end{equation}

In fact, we have that
\begin{equation}\label{eqAR11}
\lim_{n\rightarrow\infty}\frac{\ell_R(I(nD)/I(\mu_i)_{\delta n})}{n^d}=\mu{\rm Vol}(\Delta(D)).
\end{equation}

\begin{Lemma}\label{LemmaCCL} Suppose that $\Delta_1$ and $\Delta_2$ are compact, convex subsets of $\RR^d$, $\Delta_1\subset \Delta_2$ and ${\rm Vol}(\Delta_1)={\rm Vol}(\Delta_2)>0$. Then $\Delta_1=\Delta_2$.
\end{Lemma}

\begin{proof} Suppose that $\Delta_1\ne \Delta_2$. Then there exists $p\in \Delta_2\setminus \Delta_1$. Since $\Delta_1$ is closed in $\RR^d$, there exists an epsilon ball $B_{\epsilon}(p)$ centered at $p$ in $\RR^d$ such that $B_{\epsilon}(p)\cap \Delta_1=\emptyset$. Now $\Delta_2$ has positive volume, so there exist $w_1,\ldots,w_d\in \Delta_2$ such that $v_1=w_1-p,\ldots, v_d=w_d-p$  is a real basis of $\RR^d$. Since $\Delta_2$ is convex, there exists $\delta>0$ such that letting $W$ be the hypercube
$$
W=\{p+\alpha_1v_1+\cdots+\alpha_dv_d\mid 0\le \alpha_i\le \delta\mbox{ for }1\le i\le d\},
$$
we have that $W\subset \Delta_2\cap B_{\epsilon}(p)$. But then 
$$
{\rm Vol}(\Delta_2)-{\rm Vol}(\Delta_1)\ge {\rm Vol}(W)>0.
$$
a contradiction. Thus $\Delta_1=\Delta_2$.
\end{proof}

\begin{Lemma}\label{LemmaAR2} For $\delta>>0$, we have that 
${\rm Vol}(\Delta(D))>0$.
\end{Lemma}

\begin{proof} 
By (\ref{eqX5}) in the proof of Proposition \ref{PropPos}, there exists $\alpha\in \ZZ_+$ such that 
$I(\mu_i)_{\alpha n}\subset m_R^n$ for all $n\in \ZZ_+$ (since an excellent normal local ring is analytically ireducible). Further, there exists $c\in \ZZ_+$ such that $m_R^c\subset I(D)$, so that $m_R^{nc}\subset I(nD)$ for all $n$. Choosing $\delta>2\alpha$  so that 
$I(\mu_i)_{\delta n}\subset m_R^{2cn}$ for all $n$, we have that
$$
\begin{array}{lll}
{\rm Vol}(\Delta(D))&=&\frac{1}{\mu}\lim_{n\rightarrow \infty}\frac{\ell_R(nD)/I(\mu_i)_{\delta n})}{n^d}
\\&\ge &\frac{1}{\mu}\lim_{n\rightarrow \infty}\frac{\ell_R(m_R^{cn}/m_R^{2cn})}{n^d}\\
&=&\frac{1}{\mu} \frac{e_R(m_R;R)c^d(2^d-1)}{d!}>0.
\end{array}
$$

 %and so
%$$
%\lim_{n\rightarrow\infty}\frac{\ell_R(R/I(mD))}{n^d}>0
%$$
%since $I(nD)\subset I(\mu_i)_{a_in}$ for all $n\in \ZZ_+$.

\end{proof}

\begin{Theorem}\label{TheoremAR1}
Let $D_1, D_2$ be effective Cartier divisors on $X$ with exceptional support, such that
$D_1\le D_2$ and $e_R(\mathcal I_1,R)=e_R(\mathcal I_2,R)$, where $\mathcal I_1=\{I(mD_1)\}$ and $\mathcal I_2=\{I(mD_2)\}$. Then
$$
\Gamma(X,\mathcal O_X(-mD_1))=\Gamma(X,\mathcal O_X(-mD_2))
$$
for all $m\in\NN$.
\end{Theorem}

\begin{proof}
 Write $D_1=\sum_{i=1}^r a_iE_i$ and $D_2=\sum_{i=1}^rb_iE_i$ with $a_i,b_i\ge 0$ for all $i$. For each $i$ with $1\le i\le r$ choose a flag (\ref{eqAR2})
 with $Y_1=E_i$ and $p$ a closed point such that $p$ is nonsingular on $X$ and $E_i$ and $p\not\in E_j$ for $j\ne i$.  Let 
$\pi_1:\RR^{d+1}\rightarrow \RR$ be the projection onto the first factor. 

 By the definition of $\gamma_{E_i}(D_2)$ and since $\gamma_{E_i}(D_2)$ is in the closure of the compact set $\pi_1(\Delta(D_2))$,
$$
\pi_1^{-1}(\gamma_{E_i}(D_2))\cap \Delta(D_2)\ne \emptyset
$$
 and
 $$
 \pi_1^{-1}(a)\cap \Delta(D_2)=\emptyset\mbox{ if }a<\gamma_{E_i}(D_2).
 $$

We have that $D_1<D_2$ implies $\Delta(D_1)\subset \Delta(D_2)$.   We have that   ${\rm Vol}(\Delta(D_1)>0$ by Lemma \ref{LemmaAR2}. Since we are assuming that 
$e_R(\mathcal I_1;R)=e_R(\mathcal I_2;R)$, by (\ref{eqAR10}), we have that 
${\rm Vol}(D_1)={\rm Vol}(D_2)$, and so  $\Delta(D_1)=\Delta(D_2)$ by Lemma \ref{LemmaCCL}.  Thus
$$
\gamma_{E_i}(D_1)=\gamma_{E_i}(D_2)
$$
for $1\le i\le r$. We  obtain that
$$
-\sum_{i=1}^r\gamma_{E_i}(D_2)E_i=-\sum_{i=1}^r\gamma_{E_i}(D_1)E_i.
$$
By Lemma \ref{LemmaAR1}, for all $m\ge 0$,
$$
\begin{array}{lll}
\Gamma(X,\mathcal O_X(-mD_1))&=&\Gamma(X,\mathcal O_X(-\lceil \sum m\gamma_{E_i}(D_1)E_i\rceil))\\
&=&\Gamma(X,\mathcal O_X(-\lceil \sum m\gamma_{E_i}(D_2)E_i\rceil))\\
&=&\Gamma(X,\mathcal O_X(-mD_2)).\end{array}
$$
\end{proof}

We now show that Rees's theorem for $m_R$-primary ideals, \cite{R},  \cite[Proposition 11.3.1]{HS}, generalizes to divisorial filtrations, giving a converse to Theorem \ref{Theorem13} for divisorial filtrations. 

\begin{Theorem}\label{TheoremX1}  Suppose that $R$ is a $d$-dimensional excellent local domain. Let $\phi:X\rightarrow \mbox{Spec}(R)$ be the normalization of the blowup of an $m_R$-primary ideal. Suppose that $D(1)$ and $D(2)$ are effective Cartier divisors on $X$ with exceptional support such that $D(1)\le D(2)$ and $e_R(\mathcal I(1);R)=e_R(\mathcal I(2);R)$,
where $\mathcal I(1), \mathcal I(2)$ are the filtrations by $m_R$-primary ideals
$\mathcal I(1)=\{I(nD(1))\}$ and $\mathcal I(2)=\{I(nD(2))\}$. Then
$$
I(mD(1))=I(mD(2))
$$
for all $m\in \NN$.
\end{Theorem}

\begin{proof} We use the notation introduced before the statement of Lemma \ref{LemmaR1} in Section \ref{Secd-dim}. Let $D(1)_{i},D(2)_{i}$ be the divisors induced by $D(1)$ and $D(2)$ on $X_i$. Since $D(1)<D(2)$, we have that
\begin{equation}\label{eqX1}
D(1)_{i}<D(2)_{i}\mbox{ for all }i.
\end{equation}
Thus
\begin{equation}\label{eqX2*}
e_{S_{m_i}}(\{J(mD(1)_{i})\};S_{m_i})\le e_{S_{m_i}}(\{J(mD(2)_{i});S_{m_i})
\mbox{ for all $i$.}
\end{equation}
Now Lemma \ref{LemmaR1} and (\ref{eqR15}) imply
\begin{equation}\label{eqX3}
e_R(\mathcal I(j);R)=e_R(\{I(mD(j))\};R)=\sum_{i=1}^t[S/m_i:R/m_R]e_{S_{m_i}}(\{J(mD(j)_i)\};S_{m_i})
\end{equation}
for $j=1,2$.

Now the assumption $e_R(\mathcal I(1);R)=e_R(\mathcal I(2);R)$, (\ref{eqX2*}) and (\ref{eqX3}) imply
\begin{equation}\label{eqX4}
e_{S_{m_i}}(\{J(mD(1)_i)\};S_{m_i})=e_{S_{m_i}}(\{J(mD(2)_i)\};S_{m_i})
\end{equation}
for all $i$. Now (\ref{eqX1}), (\ref{eqX4}) and Theorem \ref{TheoremAR1} imply
$$
J(mD(1)_i)=\Gamma(X_i,\mathcal O_{X_i}(-mD(1)_i))=\Gamma(X_i,\mathcal O_{X_i}(-mD(2)_i))=J(mD(2)_i)
$$
for all $m\in \NN$ and all $i$. Thus 
$$
J(mD(1))=\Gamma(X,\mathcal O_X(-mD(1)))=\Gamma(X,\mathcal O_X(-mD(2)))=J(mD(2))
$$
for all $m\in \NN$ by (\ref{eqR6}). Thus
$$
I(mD(1))=J(mD(1))\cap R=J(mD(2))\cap R=I(mD_2)
$$
for all $m\in \NN$.
\end{proof}

\section{A Geometric Rees Theorem}

Let $X$ be a normal projective variety over a field $k$ of dimension $d$. Suppose that $D$ is an effective Cartier divisor on $X$. The volume of $D$ is
$$
{\rm Vol}(D)=\lim_{m\rightarrow\infty}\frac{\dim_k\Gamma(X,\mathcal O_X(mD))}{m^d/d!}.
$$
Let $E$ be a codimension one prime divisor on $X$. For $m\in \NN$, define
$$
\tau_{E,m}(D)=\min\{\mbox{ord}_E\Delta\mid \Delta\in |mD|\}.
$$
Let $\tau_i=\tau_{E,i}(D)$. Then since $\tau_{mn}\le n\tau_m$, we have that
\begin{equation}\label{eqGR1}
\frac{\tau_{mn}}{mn}\le\min\{\frac{\tau_m}{m},\frac{\tau_n}{n}\}.
\end{equation}

Now define
$$
\gamma_E(D)=\inf_m\frac{\tau_m}{m}.
$$

Expand $D=\sum_{i=1}^r a_iE_i$ with $E_i$ prime divisors and $a_i\in \ZZ_+$. 

\begin{Lemma}\label{LemmaGR1}
We have that
$$
\Gamma(X,\mathcal O_X(mD))=\Gamma(X,\mathcal O_X(mD-\sum_{i=1}^r\lceil m\gamma_{E_i}(D)\rceil E_i))
=\Gamma(X,\mathcal O_X(\lfloor mD-\sum_{i=1}^r m\gamma_{E_i}(D)E_i\rfloor))
$$
for all $m\in \NN$.
\end{Lemma}

\begin{proof} Suppose that $\Delta\in |mD|$. Then $\Delta-\sum_i\tau_{E_i,m}(D)E_i\ge 0$ so that $\Delta-\sum m\gamma_{E_i}E_i\ge 0$. Thus $\Delta-\sum_{i=1}^r \lceil m\gamma_{E_i}(D)\rceil E_i\ge 0$. 
\end{proof}

We now recall the method of \cite{LM} to compute volumes of Cartier divisors, as extended in \cite{C2} to arbitrary fields.
Suppose that $p\in X$ is a nonsingular closed point and 
\begin{equation}\label{eqGR2}
X=Y_0\supset Y_1\supset \cdots \supset Y_d=\{p\}
\end{equation}
is a flag; that is, the $Y_i$ are subvarieties of $X$ of dimension $d-i$ such that there is a regular system of parameters $a_1,\ldots,a_d$ in $\mathcal O_{X,p}$ such that  $a_1=\cdots=a_i=0$ are local equations of $Y_i$ in $X$ for $1\le i\le d$. 

The flag determines a valuation $\nu$ on the function field $k(X)$ of $X$  as follows.  We have a sequence of natural surjections of regular local rings
\begin{equation}\label{eqGR3} \mathcal O_{X,p}=
\mathcal O_{Y_0,p}\overset{\sigma_1}{\rightarrow}
\mathcal O_{Y_1,p}=\mathcal O_{Y_0,p}/(a_1)\overset{\sigma_2}{\rightarrow}
\cdots \overset{\sigma_{d-1}}{\rightarrow} \mathcal O_{Y_{d-1},p}=\mathcal O_{Y_{d-2},p}/(a_{d-1}).
\end{equation}
Define a rank $d$ discrete valuation $\nu$ on $k(X)$   by prescribing for $s\in \mathcal O_{X,p}$,
$$
\nu(s)=({\rm ord}_{Y_1}(s),{\rm ord}_{Y_2}(s_1),\cdots,{\rm ord}_{Y_d}(s_{d-1}))\in (\ZZ^d)_{\rm lex}
$$
where 
$$
s_1=\sigma_1\left(\frac{s}{a_1^{{\rm ord}_{Y_1}(s)}}\right),
s_2=\sigma_2\left(\frac{s_1}{a_2^{{\rm ord}_{Y_2}(s_1)}}\right),\ldots,
s_{d-1}=\sigma_{d-1}\left(\frac{s_{d-2}}{a_{d-1}^{{\rm ord}_{Y_{d-1}}(s_{d-2})}}\right).
$$

let $g=0$ be a local equation of $D$ at $p$. For $m\in \NN$, define
$$
\Phi_{mD}:\Gamma(X,\mathcal O_X(mD))=\{f\in k(X)\mid (f)+mD\ge 0\}\rightarrow \ZZ^d
$$
by $\Phi_{mD}(f)=\nu(fg^m)$. The Newton Okounkov body $\Delta(D)$ of $D$ is the closure of the set 
$$
\cup_{m\in \NN}\frac{1}{m}\Phi_{mD}(\Gamma(X,\mathcal O_X(mD)))
$$
in $\RR^d$. This is a compact and convex set by \cite[Lemma 1.10]{LM} or the proof of Theorem 8.1 \cite{C2}.

Modifying the proof of \cite[Theorem 8.1]{C2} and of \cite[Lemma 5.4]{C3} we see that
\begin{equation}\label{GR4}
{\rm Vol}(D)=\lim_{m\rightarrow \infty}\frac{\dim_k\Gamma(X,\mathcal O_X(mD))}{m^d/d!}
=d![\mathcal O_{X,p}/m_p:k]{\rm Vol}(\Delta(D)).
\end{equation}

Suppose that $D_1<D_2$ are effective Cartier divisors on $X$. Let $g_1=0$ be a local equation of $D_1$ at $p$, $g_2=0$ be a local equation of $D_2$ at $p$, so that $h=\frac{g_2}{g_1}$ is a local equation of $D_2-D_1$ at $p$. We have  commutative diagrams
$$
\begin{array}{lll}
\Gamma(X,\mathcal O_D(mD_1))&\rightarrow &\Gamma(X,\mathcal O_X(mD_2))\\
\downarrow\Phi_{mD_1}\times\{m\}&&\downarrow \Phi_{mD_2}\times\{m\}\\
\ZZ^{d+1}&\rightarrow& \ZZ^{d+1}
\end{array}
$$
where the top horizontal arrow is the natural inclusion and the bottom horizontal arrow is the map
$$
(\alpha,m)\mapsto  (\alpha+m\nu(h),m).
$$
These diagrams induce an inclusion $\Lambda:\Delta(D_1)\rightarrow \Delta(D_2)$ defined by 
$\alpha\mapsto \alpha+\nu(h)$.

\begin{Theorem}\label{TheoremGRT} Suppose that $X$ is a normal projective variety over a field $k$ and $D_1,D_2$ are effective Cartier divisors on $X$ such that $D_1$ is big,  $D_1\le D_2$ and ${\rm Vol}(D_1)={\rm Vol}(D_2)$. Then
$$
\Gamma(X,\mathcal O_X(nD_1))=\Gamma(X,\mathcal O_X(nD_2))
$$
for all $n\in \NN$.
\end{Theorem}

\begin{proof} Write $D_1=\sum_{i=1}^r a_iE_i$ and $D_2=\sum_{i=1}^rb_iE_i$ with $a_i,b_i\ge 0$ for all $i$. For each $i$ with $1\le i\le r$ choose a flag (\ref{eqGR2})
 with $Y_1=E_i$ and $p$ a point such that $p\in X$ is a nonsingular closed point of $X$ and $E_i$ and  $p\not\in E_j$ for $j\ne i$.  Let 
$\pi_1:\RR^d\rightarrow \RR$ be the projection onto the first factor. Then with the notation intoduced above, $\nu(h)=(b_i-a_i,0,\ldots,0)$. By the definition of $\gamma_{E_i}(D_2)$ and since $\gamma_{E_i}(D_2)$ is in the closure of the compact set $\pi_1(\Delta(D_2))$, we have that 
$$
\pi_1^{-1}(\gamma_{E_i}(D_2))\cap \Delta(D_2)\ne \emptyset
$$
 and
 $$
 \pi_1^{-1}(a)\cap \Delta(D_2))=\emptyset\mbox{ if }a<\gamma_{E_i}(D_2).
 $$
  Further, $\Lambda(\Delta(D_1))\subset \Delta(D_2)$ and ${\rm Vol}(D_1)={\rm Vol}(D_2)$, so $\Lambda(\Delta(D_1))=\Delta(D_2)$ by Lemma \ref{LemmaCCL}.
  Thus
$$
\gamma_{E_i}(D_1)=\gamma_{E_i}(D_2)-(b_i-a_i)
$$
for $1\le i\le r$. We  obtain that
$$
D_2-\sum_{i=1}^r\gamma_{E_i}(D_2)E_i=D_1-\sum_{i=1}^r\gamma_{E_i}(D_1)E_i.
$$
By Lemma \ref{LemmaGR1}, for all $m\ge 0$,
$$
\begin{array}{lll}
\Gamma(X,\mathcal O_X(mD_1))&=&\Gamma(X,\mathcal O_X(\lfloor mD_1-\sum m\gamma_{E_i}(D_1)E_i\rfloor))\\
&=&\Gamma(X,\mathcal O_X(\lfloor mD_2-\sum m\gamma_{E_i}(D_2)E_i\rfloor))\\
&=&\Gamma(X,\mathcal O_X(mD_2)).\end{array}
$$
\end{proof}

\section{Mixed Multiplicities of two dimensional Excellent local rings}\label{Sec5}

\subsection{2-dimensional normal local rings}
In this subsection, suppose that $R$ is an excellent, normal local ring of dimension two, so that  $R$ is analytically irreducible. 
Resolutions of singularities of $\mbox{Spec}(R)$ exist by \cite{L2} or \cite{CJS}.
 Let $\phi:X\rightarrow \mbox{Spec}(R)$ be a resolution of singularities with prime (integral) exceptional curves $E_1,\ldots,E_s$.
By  \cite[Lemma 14.1]{L1}, the intersection matrix of $E_1,\ldots,E_s$ is negative definite. 
Thus there exists an effective (necessarily Cartier) divisor $B$ on $X$ with exceptional support such that $\mathcal O_X(-B)$ is very ample, and  so $\phi$ is the blowup of the $m_R$-primary ideal $\phi_*\mathcal O_X(-B)$.

We refer to \cite{L1} for background material for this section. 
A $\QQ$-divisor on $X$ with exceptional support is a formal linear combination of prime exceptional curves with rational coefficients. A $\QQ$-divisor $C$ is anti-nef if $(C\cdot E)\le 0$ for all exceptional curves $E$ on $X$. Suppose that $f\in {\rm QF}(R)$. Then $(f)$ will denote the divisor of $f$ on $X$.

\begin{Lemma}\label{LemmaV1}
Let $D$ be an effective divisor on $X$ with exceptional support. Then there is a unique minimal effective anti-nef $\QQ$-divisor $\Delta$ on  $X$ with exceptional support such that $D\le \Delta$.

The $\QQ$-divisor 
$\Delta$ is the unique effective $\QQ$-divisor $\Delta$ on $X$ such that
\begin{enumerate}
\item[1)] $\Delta=D+B$ is anti-nef and $B$ is effective.
\item[2)] $(\Delta\cdot E)=0$ if $E$ is a component of $B$.
\end{enumerate}
\end{Lemma}
 The first sentence of the lemma  follows from  the proof of the existence of Zariski decomposition in  \cite{B}.
The second sentence is the local formulation \cite[Proposition 2.1]{CHR}  of the  classical theorem of Zariski \cite{Z}.

We will say that the expression 1) is the Zariski decomposition of $D$ and that $\Delta$ is the anti-nef part of the Zariski decomposition of $D$.

\begin{Remark}\label{Remarkv1}From the first sentence of the lemma, we deduce that if $D_1\le D_2$ are effective divisors with exceptional support and respective anti-nef parts of their Zariski decompositions $\Delta_1$ and $\Delta_2$, then $\Delta_1\le\Delta_2$ as necessarily $D_1\le \Delta_2$.
\end{Remark}

\begin{Corollary}\label{CorV1}
Suppose that $D_1\le D_2$ are effective divisors with effective support, and respective anti-nef parts of their Zariski decompositions $\Delta_1$ and $\Delta_2$.  Then $(\Delta_2^2)\le (\Delta_1^2)$ with equality if and only if 
$\Delta_1=\Delta_2$.
\end{Corollary}

\begin{proof} If $\Delta$ is an anti-nef divisor with exceptional support, and $E$ is a nonzero effective $\QQ$-divisor with exceptional support, then 
$$
(\Delta+E)^2=(\Delta^2)+2(\Delta\cdot E)+(E^2)<(\Delta^2)
$$
since $(E^2)<0$ as the intersection form on exceptional divisors on $X$ is negative definite.
\end{proof}

Let $\nu_i$ be the discrete valuation with valuation ring $\mathcal O_{X,E_i}$ for $1\le i\le r$, and define the valuation ideals
$$
I(\nu_i)_n=  \{f\in R\mid \nu_i(f)\ge n\}
$$
for $n\in \NN$ and $1\le i\le r$.

For $D=a_1E_1+\cdots+a_rE_r$ an effective  integral divisor on $X$ with exceptional support ($a_i\in \NN$ for all $i$), define 
$$
I(D)=\Gamma(X,\mathcal O_X(-D))=\{f\in {\rm QF}(R)\mid (f)-D\ge 0\}.
$$
We have that  $I(0)=\Gamma(X,\mathcal O_X)=R$  since the ring $\Gamma(X,\mathcal O_X)$ is a finitely generated $R$-module with the same quotient field as $R$ and $R$ is normal. Thus $I(D)$ is an $m_R$-primary ideal if $D\ne 0$. For $n\in \NN$, we have that
$$
I(nD)=I(\nu_1)_{na_1}\cap\cdots\cap I(\nu_r)_{na_r}
$$
is an $m_R$-primary ideal in $R$, and $\{I(nD)\}$ is a filtration of $m_R$-primary ideals in $R$. By Theorem \ref{TheoremI20}, the limit
$$
{\rm Vol}(D):=\lim_{n\rightarrow \infty}\frac{\ell_R(R/I(nD))}{n^2/2!}=e_R(\{I(nD)\};R)
$$
exists. In fact, by formula (7) and Lemma 2.5 on page 6 of \cite{CHR}, we have
\begin{equation}\label{eqV1}
{\rm Vol}(D)=-(\Delta^2)
\end{equation}

where $\Delta$ is the anti-nef part of the Zariski decomposition of $D$. 

\begin{Remark}\label{RemarkV2}We deduce from Corollary \ref{CorV1} that if $D_1\le D_2$ are effective divisors with exceptional support on $X$ and respective anti-nef parts of their Zariski decompositons $\Delta_1$ and $\Delta_2$, then
$$
{\rm Vol}(D_1)\le {\rm Vol}(D_2)
$$
with equality if and only if 
$\Delta_1=\Delta_2$.
\end{Remark}

Let $\lceil a\rceil$ denote the smallest integer which is greater than or equal to a real number $a$. If $D=\sum a_iE_i$ with $a_i\in \QQ$ is a $\QQ$-divisor, let $\lceil D\rceil = \sum \lceil a_i\rceil E_i$.

\begin{Lemma}\label{LemmaV2}
Suppose that $D$ is an effective divisor on $X$ with exceptional support and $\Delta=D+B$ is the Zariski decomposition of $D$. Then for all $n\in \NN$, $I(nD)=I(\lceil n\Delta\rceil)$.
\end{Lemma}

\begin{proof} Suppose that $f\in I(\lceil n\Delta\rceil)=\Gamma(X,\mathcal O_X(-\lceil n\Delta\rceil))$. Then $(f)-\lceil n\Delta\rceil\ge 0$. Writing  $n\Delta=\lceil n\Delta\rceil -G$ with $G\ge 0$, we have $-n\Delta=G-\lceil n\Delta\rceil$.
From 
$$
-nD=-n\Delta+nB=-\lceil n\Delta\rceil+(G+nB)
$$
and the fact that $G+nB\ge 0$, we have that $(f)-nD\ge 0$ so that $f\in \Gamma(X,\mathcal O_X(-nD))=I(nD)$.

Let $S$ be the set of irreducible curves in the support of $B$. Suppose that $f\in I(nD)=\Gamma(X,\mathcal O_X(-nD))$. Then $(f)-nD\ge 0$. Write $(f)-nD=A+C$ where $A$ and $C$ are effective divisors on $X$, no components of $A$ are in $S$ and all components of $C$ are in $S$. We have that $(f)-n\Delta=A+(C-nB)$. If $E\in S$ then 
$$
(E\cdot(A+(C-nB)))=(E\cdot((f)-n\Delta))=0
$$
which implies $(E\cdot (C-nB))=-(E\cdot A)\le 0$. The intersection matrix of the curves in $S$ is negative definite since it is so for the set of all exceptional curves, so $C-nB\ge 0$ (for instance by \cite[Lemma 14.0]{Ba}). Thus $(f)-n\Delta\ge 0$ which implies $(f)-\lceil n\Delta\rceil\ge 0$ since $(f)$ is an integral divisor (that is, has integral coefficients).  Thus $f\in \Gamma(X,\mathcal O_X(-\lceil n\Delta\rceil))=I(\lceil n\Delta\rceil)$.  
\end{proof}

\begin{Proposition}\label{PropV2} Suppose that $D_1$ and $D_2$ are effective divisors with exceptional support on $X$. Let $\mathcal I(1)=\{I(nD_1)\}$ and $\mathcal I(2)=\{I(nD_2)\}$.  Suppose that $D_1\le D_2$ and
$$
e_R(\mathcal I(1);R)=e_R(\mathcal I(2);R).
$$
Then $I(nD_1)=I(nD_2)$ for all $n\in \NN$.
\end{Proposition}

\begin{proof} Let $\Delta_1$ and $\Delta_2$ be the respective anti-nef parts of the Zariski decompositions of $D_1$ and $D_2$. By Remark \ref{RemarkV2}, $D_1\le D_2$ and ${\rm Vol}(D_1)={\rm Vol}(D_2)$ implies $\Delta_1=\Delta_2$. Thus 
$$
I(nD_1)=I(\lceil n\Delta_1\rceil)=I(\lceil n\Delta_2\rceil)=I(nD_2)
$$
for all $n\in \NN$ by Lemma \ref{LemmaV2}.
\end{proof}

\begin{Proposition}\label{PropV1} Suppose that $D_1,\ldots,D_r$ are effective divisors on $X$ with exceptional support. For $n_1,\ldots,n_r\in \NN$,
let
$$
G(n_1,\ldots,n_r)=\lim_{n\rightarrow \infty}\frac{\ell_R(R/I(nn_1D_1)\cdots I(nn_rD_r))}{n^2}.
$$
Then for $n_1,\ldots,n_r\in \NN$,
$$
G(n_1,\ldots,n_r)=-\frac{1}{2}((n_1\Delta_1+n_2\Delta_2+\cdots+n_r\Delta_r)^2)
$$
where $\Delta_1,\ldots,\Delta_r$ are the respective anti-nef parts of the Zariski decompositions of $D_1,\ldots,D_r$.
\end{Proposition}

\begin{proof}  Fix $n_1,\ldots,n_r\in\NN$. Given $\epsilon>0$, there exist effective $\QQ$-divisors $F_{1,\epsilon}, \ldots, F_{r,\epsilon},A_{1,\epsilon},\ldots,A_{r,\epsilon}$ with exceptional support such that $-A_{i,\epsilon}$  are ample for $1\le i\le r$ (that is, $(A_{i,\epsilon}\cdot E)<0$ for all exceptional curves $E$ and $(A_{i,\epsilon}^2)>0$), 
$-n_i\Delta_i=-A_{i,\epsilon}+F_{i,\epsilon}$ for $1\le i\le r$,
$$
|((n_1\Delta_1+\cdots+n_r\Delta_r)^2)-((A_{1,\epsilon}+\cdots+A_{r,\epsilon})^2)|<\epsilon
$$
and
$$
|(n_i\Delta_i^2)-(A_{i,\epsilon}^2)|<\epsilon\mbox{ for }1\le i\le r.
$$
Let
$A_{\epsilon}=A_{1,\epsilon}+\cdots+A_{r,\epsilon}$, $F_{\epsilon}=F_{1,\epsilon}+\dots+F_{r,\epsilon}$ so that
$$
-(n_1\Delta_1+\cdots+n_r\Delta_r)=-A_{\epsilon}+F_{\epsilon}.
$$

There exists $s_{\epsilon}\in \ZZ_+$ such that $s_{\epsilon}A_{i,\epsilon}$ and $s_{\epsilon}\Delta_i$ are effective integral divisors (that is, have integral coefficients) for $1\le i\le r$. Since the $-s_{\epsilon}A_{i,\epsilon}$ are ample integral divisors on $X$, there exists $\alpha_{\epsilon}\in \ZZ_+$ such that 
the invertible sheaves $\mathcal O_X(-\alpha_{\epsilon}s_{\epsilon}A_{i,\epsilon})$ are generated by global sections for $1\le i\le r$. Thus for $n\in\NN$,
$$
\begin{array}{lll}
I(\alpha_{\epsilon}s_{\epsilon}A_{1,\epsilon})^n\cdots I(\alpha_{\epsilon}s_{\epsilon}A_{r,\epsilon})^n\mathcal O_X 
&=&
I(n\alpha_{\epsilon}s_{\epsilon}A_{1,\epsilon})\cdots I(n\alpha_{\epsilon}s_{\epsilon}A_{r,\epsilon})\mathcal O_X \\
&=& I(n\alpha_{\epsilon}s_{\epsilon}A_{\epsilon})\mathcal O_X=I(\alpha_{\epsilon}s_{\epsilon}A_{\epsilon})^n\mathcal O_X.
\end{array}
$$

Thus  the ideals 
$$
I(\alpha_{\epsilon}s_{\epsilon}A_{1,\epsilon})^n\cdots I(\alpha_{\epsilon}s_{\epsilon}A_{r,\epsilon})^n,
I(n\alpha_{\epsilon}s_{\epsilon}A_{1,\epsilon})\cdots I(n\alpha_{\epsilon}s_{\epsilon}A_{r,\epsilon}),
I(n\alpha_{\epsilon}s_{\epsilon}A_{\epsilon}), I(\alpha_{\epsilon}s_{\epsilon}A_{\epsilon})^n
$$
have the same integral closure which is $I(n\alpha_{\epsilon}s_{\epsilon}A_{\epsilon})$, and so the $R$-algebra
$$
\bigoplus_{n\ge 0} I(n\alpha_{\epsilon}s_{\epsilon}A_{\epsilon})
$$
is integral over
$$
\bigoplus_{n\ge 0}I(n\alpha_{\epsilon}s_{\epsilon}A_{1,\epsilon})\cdots I(n\alpha_{\epsilon}s_{\epsilon}A_{r,\epsilon}).
$$
Now by Theorem \ref{Theorem13}  and  (\ref{eqV1}),
\begin{equation}\label{eqV2}
\begin{array}{lll}
\lim_{n\rightarrow \infty}\frac{\ell_R(R/I(n\alpha_{\epsilon}s_{\epsilon}A_{1,\epsilon})\cdots I(n\alpha_{\epsilon}s_{\epsilon}A_{r,\epsilon}))}{n^2}
&=&\lim_{n\rightarrow\infty}\frac{\ell_R(R/I(n\alpha_{\epsilon}s_{\epsilon}A_{\epsilon}))}{n^2}\\
&=&-\frac{1}{2}((\alpha_{\epsilon}s_{\epsilon}A_{\epsilon})^2)
=-\frac{\alpha_{\epsilon}^2s_{\epsilon}^2}{2}(A_{\epsilon}^2).
\end{array}
\end{equation}
For all $n\in\NN$, we have inclusions
$$
I(n\alpha_{\epsilon}s_{\epsilon}A_{1,\epsilon})\cdots I(n\alpha_{\epsilon}s_{\epsilon}A_{r,\epsilon})\subset 
I(n\alpha_{\epsilon}s_{\epsilon}n_1\Delta_1)\cdots I(n\alpha_{\epsilon}s_{\epsilon}n_r\Delta_r))\subset I( n\alpha_{\epsilon}s_{\epsilon}(n_1\Delta_1+\cdots+n_r\Delta_r))
$$
inducing surjections
$$
\begin{array}{l}
R/I( n\alpha_{\epsilon}s_{\epsilon}A_{1,\epsilon})\cdots I( n\alpha_{\epsilon}s_{\epsilon}A_{r,\epsilon}))\rightarrow R/I(n\alpha_{\epsilon}s_{\epsilon}n_1\Delta_1)\cdots I( n\alpha_{\epsilon}s_{\epsilon}n_r\Delta_r))\\
\rightarrow R/I(n\alpha_{\epsilon}s_{\epsilon}(n_1\Delta_1+\cdots+n_r\Delta_r))
\end{array}
$$
so that
$$
\begin{array}{lll}
-\frac{1}{2}(A_{\epsilon}^2)&=&
\frac{1}{\alpha_{\epsilon}^2s_{\epsilon}^2}\lim_{n\rightarrow\infty}\frac{\ell_R(R/I(n\alpha_{\epsilon}s_{\epsilon}A_{1,\epsilon})\cdots I(n\alpha_{\epsilon}s_{\epsilon}A_{r,\epsilon}))}{n^2}\\
&\ge& \frac{1}{\alpha_{\epsilon}^2s_{\epsilon}^2}\lim_{n\rightarrow\infty}\frac{\ell_R(R/I(n\alpha_{\epsilon}s_{\epsilon}n_1\Delta_1)\cdots I(n\alpha_{\epsilon}s_{\epsilon}n_r\Delta_r))}{n^2}\\
&\ge& \frac{1}{\alpha_{\epsilon}^2s_{\epsilon}^2}\lim_{n\rightarrow\infty}\frac{\ell_R(R/I(n\alpha_{\epsilon}s_{\epsilon}(n_1\Delta_1+\cdots+n_r\Delta_r)))}{n^2}\\
&=&\frac{1}{\alpha_{\epsilon}^2s_{\epsilon}^2}\left[-\frac{1}{2}((\alpha_{\epsilon}s_{\epsilon}(n_1\Delta_1+\cdots+n_r\Delta_r))^2)\right]\\
&=& -\frac{1}{2}((n_1\Delta_1+\cdots+n_r\Delta_r)^2).
\end{array}
$$

Now
$$
\begin{array}{lll}
\lim_{n\rightarrow\infty}\frac{\ell_R(R/I(n\alpha_{\epsilon}s_{\epsilon}n_1\Delta_1)\cdots I(n\alpha_{\epsilon} s_{\epsilon}n_r\Delta_r))}{n^2}
&=&
\lim_{n\rightarrow\infty}\frac{\ell_R(R/I(n\alpha_{\epsilon}s_{\epsilon}n_1D_1)\cdots I(n\alpha_{\epsilon}s_{\epsilon}n_rD_r))}{n^2}\\
&=& (\alpha_{\epsilon}^2s_{\epsilon}^2)\lim_{n\rightarrow \infty}\frac{\ell_R(R/I(nn_1D_1)\cdots I(nn_rD_r))}{n^2}.
\end{array}
$$

Thus 
\begin{equation}\label{eqV3}
\begin{array}{lll}
-\frac{1}{2}((n_1\Delta_1+\cdots+n_r\Delta_r)^2)&=&\lim_{\epsilon\rightarrow 0}-\frac{1}{2}(A_{\epsilon}^2)\\
%&=&\lim_{n\rightarrow\infty}\frac{\ell_R(R/I(\lceil nn_1\Delta_1\rceil)\cdots I(\lceil nn_r\Delta_r\rceil))}{n^2}\\
&=&\lim_{n\rightarrow\infty}\frac{\ell_R(R/I(nn_1D_1)\cdots I(nn_rD_r))}{n^2}\\
&=&G(n_1,\ldots,n_r).
\end{array}
\end{equation}
\end{proof}

From Proposition \ref{PropV1} and equation (\ref{eqV6}), with $\mathcal I(i)=\{I(nD_i)\}$, we deduce that the mixed multiplicities are
\begin{equation}\label{eqV7}
e_R(\mathcal I(j)^{[2]};R)=-(\Delta_j^2)\mbox{ for all $j$}
\end{equation}
and
\begin{equation}\label{eqV8}
e_R(\mathcal I(i)^{[1]},\mathcal I(j)^{[1]};R)=-(\Delta_i\cdot\Delta_j)
\end{equation}
for $i\ne j$.

We have by Proposition \ref{PropPos}  (or since  $-(\Delta_j^2)>0$ for all $j$ since $\Delta_j\ne 0$ and the intersection form is negative definite) that  all mixed multiplicities are positive. Further, the mixed multiplicities are all rational numbers since the $\Delta_i$ are $\QQ$-divisors. 

%%%%%%%%%%%%%%%%%
\subsection{two-dimensional local domains}
We now assume that $R$ has dimension two and $X$ is nonsingular. 
We use the notation introduced before the statement of Lemma \ref{LemmaR1} in Section \ref{Secd-dim}.

For $1\le l\le r$, write $D(l)=\sum_{i,j}a_{i,j}(l)E_{i,j}$ with $a_{i,j}\in \NN$ and let $D(l)_i=\sum_ja_{i,j}(l)E_{i,j}$. Let $\Delta(l)_i$ be the anti-nef part of the Zariski decomposition of $D(l)_i$. 
For $n_1,\ldots,n_r\in\NN$, 
$$
\begin{array}{lll}
\lim_{n\rightarrow \infty}\frac{\ell_R(S/J(nn_1D(1)\cdots J(nn_rD(r))}{n^2}&=&
\sum_{i=1}^t\lim_{n\rightarrow \infty} \frac{\ell_R(S_{m_i}/J(nn_1D(1)_i)\cdots J(nn_r D(r)_i))}{n^2}\\
&=& \sum_{i=1}^t-\frac{1}{2}[S/m_i:R/m_R]((n_1\Delta(1)_i+\cdots+n_r\Delta(r)_i)^2)
\end{array}
$$
by (\ref{eqR15}) and Proposition \ref{PropV1}. Now by Lemma \ref{LemmaR1} and the multinomial theorem,
\begin{equation}\label{eqR2}
\begin{array}{l}
\lim_{n\rightarrow \infty}\frac{\ell_R(R/I(nn_1D)\cdots I(nn_rD_r))}{n^2}
= \sum_{i=1}^t-\frac{1}{2}[S/m_i:R/m_R]((n_1\Delta(1)_i+\cdots+n_r\Delta(r)_i)^2)\\
=  \sum_{k_1+\ldots+k_r=2}\frac{1}{k_1!\cdots k_r!}\left(\sum_{i=1}^t-[S/m_i:R/m_R]
(\Delta(1)_i^{k_1}\cdot \ldots\cdot \Delta(r)_i^{k_r})\right)n_1^{k_1}\cdots n_r^{k_r}
\end{array}
\end{equation}

Let  $\mathcal I(i)=\{I(nD(i))\}$ be the filtrations of $m_R$-primary ideals. 
Then  by (\ref{eqV6}), the mixed multiplicities are
\begin{equation}\label{eqR13}
e_R(\mathcal I(j)^{[2]};R)=\sum_{i=1}^t-[S/m_i:R/m_R](\Delta(j)_i^2)
\end{equation}
 and for $j\ne k$,
 \begin{equation}\label{eqR14}
 e_R(\mathcal I(j)^{[1]},\mathcal I(k)^{[1]};R)=\sum_{i=1}^t-[S/m_i:R/m_R](\Delta(j)_i\cdot\Delta(k)_i).
 \end{equation}

\begin{Proposition}\label{PropR10} Suppose that $R$ is a two-dimensional excellent local domain, $\phi:X\rightarrow \mbox{Spec}(R)$ is a resolution of singularities and that $D(1)$ and $D(2)$ are effective divisors with exceptional support on $X$. Let $\mathcal I(1)=\{I(nD(1))\}$ and $\mathcal I(2)=\{I(nD(2))\}$ be the associated filtrations of $m_R$-primary ideals.  Suppose that $D(1)\le D(2)$ and
$$
e_R(\mathcal I(1);R)=e_R(\mathcal I(2);R).
$$
Then $I(nD(1))=I(nD(2))$ for all $n\in \NN$.
\end{Proposition}

\begin{proof} 

%Let $\mathcal J(1)=\{J(nD(1))\}$ and $\mathcal J(2)=\{J(nD(2))\}$ 

Let $\Delta(1)_i$ and $\Delta(2)_i$ be the respective anti-nef parts of the Zariski decompositions of $D(1)_i$ and $D(2)_i$. Then $D(1)_i\le D(2)_i$ and so $\Delta(1)_i\le\Delta(2)_i$ for all $i$, by Remark \ref{Remarkv1}. Thus by Corollary \ref{CorV1}, for all $i$, $(\Delta(2)_i^2)\le (\Delta(1)_i^2)$ with equality if and only if $\Delta(1)_i=\Delta(2)_i$. Since $e_R(\mathcal I(1);R)=e_R(\mathcal I(2);R)$,  equation (\ref{eqR13})  and (\ref{eqX30}) imply that 

$$
\sum_{i=1}^t[S/m_i:R/m_R][(\Delta(2)_i^2)-(\Delta(1)_i^2)]=0.
$$
Thus $\Delta(2)_i=\Delta(1)_i$ for all $i$, which implies that $J(nD(1)_i)=J(nD(2)_i)$ for all $n\in \NN$ by Lemma \ref{LemmaV2} and so $J(nD(1))=J(nD(2))$ for all $n$ by (\ref{eqR6}).
Thus 
 $$
 I(nD(2))=J(nD(2))\cap R=J(nD(1))\cap R=I(nD(1))
 $$
 for all $n\in \NN$.
\end{proof}

Theorem \ref{TheoremX1} in the case that $\dim R=2$ is an immediate corollary of Proposition \ref{PropR10}.
The following theorem is a generalization to divisorial valuations of a theorem of Teissier \cite{T3} and Rees and Sharp \cite{RS} for $m_R$-primary ideals.

\begin{Theorem}\label{PropR11} Suppose that $R$ is a two-dimensional excellent local domain, $\phi:X\rightarrow \mbox{Spec}(R)$ is a resolution of singularities   and that $D(1)$ and $D(2)$ are effective divisors with exceptional support on $X$. Let $\mathcal I(1)=\{I(nD(1))\}$ and $\mathcal I(2)=\{I(nD(2))\}$ be the associated filtrations of $m_R$-primary ideals. Suppose that  the Minkowski equality 
\begin{equation}\label{eqR4}
e_R(\mathcal I(1)\mathcal I(2);R)^{\frac{1}{2}}= e_R(\mathcal I(1);R)^{\frac{1}{2}}+e_R(\mathcal I(2);R)^{\frac{1}{2}}
\end{equation}
holds (there is equality in inequality 4) of  Theorem \ref{TheoremMI}).
 Then there exist relatively prime $a,b\in \ZZ_+$ such that 
$$
I(naD(1))=I(nbD(2))
$$
for all $n\in \NN$.
\end{Theorem}

\begin{proof}  We will use the notation introduced before the statement of Lemma \ref{LemmaR1}. Let $e_0=e_R(\mathcal I(1)^{[2]};R)$, $e_1=e_R(\mathcal I(1)^{[1]},\mathcal I(2)^{[1]};R)$ and $e_2=e_R(\mathcal I(2)^{[2]};R)$. Let $\Delta(1)_i$ and $\Delta(2)_i$ be the respective anti-nef parts of the Zariski decompositions of $D(1)_i$ and $D(2)_i$. Let
$$
G(n_1,n_2)=\lim_{n\rightarrow \infty}\frac{\ell_R(R/I(nn_1D(1))I(nn_2D(2)))}{n^2}.
$$
Then 
$$
G(n_1,n_2)=\frac{1}{2}e_0n_1^2+e_1n_1n_2+\frac{1}{2}e_2n_2^2
$$
by (\ref{eqV6}). Now by (\ref{eqR13}) and (\ref{eqR14}),
$$
e_0=\sum_{i=1}^t-[S/m_i:R/m_R](\Delta(1)_i^2), 
e_1=\sum_{i=1}^t-[S/m_i:R/m_R](\Delta(1)_i\cdot\Delta(2)_i),
$$
$$
e_2=\sum_{i=1}^t-[S/m_i:R/m_R](\Delta(2)_i^2).
$$

We have the Minkowski inequality (inequality 1) of Theorem \ref{TheoremMI})
\begin{equation}\label{eqR5}
e_1^2\le e_0e_2.
\end{equation}
We conclude that 
$$
e_R(\mathcal I(1)\mathcal I(2);R)=2G(1,1)=e_0+2e_1+e_2\le e_0+2e_0^{\frac{1}{2}}e_2^{\frac{1}{2}}+e_2
=(e_0^{\frac{1}{2}}+e_2^{\frac{1}{2}})^2.
$$
We deduce that equality holds in (\ref{eqR4}) if and only if equality holds in (\ref{eqR5}).
Since we assume equality in (\ref{eqR4}), we have equality in (\ref{eqR5}).
 Write
$$
\frac{e_1}{e_0}=\frac{e_2}{e_1}=\frac{a}{b}
$$
with $a,b\in \ZZ_+$ relatively prime. Replacing $D(1)$ with $aD(1)$ and $D(2)$ with $bD(2)$ we obtain $e_0=e_1=e_2$ so
$$
\sum_{i=1}^t-[S/m_i:R/m_R](\Delta(1)_i^2)=
\sum_{i=1}^t-[S/m_i:R/m_R](\Delta(1)_i\cdot\Delta(2)_i)=
\sum_{i=1}^t-[S/m_i:R/m_R](\Delta(2)_i^2).
$$

We have that 
$$
\sum_{i=1}^t[S/m_i:R/m_R]((\Delta(1)_i-\Delta(2)_i)^2)=\sum_{i=1}^t[S/m_i:R/m_R][(\Delta(1)_i^2)-2(\Delta(1)_i\cdot\Delta(2)_i)+(\Delta_2^2)]=0
$$
which implies  that $\Delta(1)_i=\Delta(2)_i$ for all $i$ since the intersection product is negative definite, so $J(nbD(1)_i)=J(naD(2)_i)$ for all $i$ and $n\in \NN$ by Lemma \ref{LemmaV2}, and thus $J(naD(1))=J(nbD(2))$ for all $n\in \NN$ by (\ref{eqR6}). Now 
$$
I(naD(1))=J(naD(1))\cap R=J(nbD(2))\cap R=I(nbD(2))
$$
for all $n\in \NN$. 

\end{proof}

%%%%%%%%%%%%%%%%%%%%%%%%%%%%%%
\begin{Corollary}\label{Prop12} Suppose that $R$ is a two-dimensional excellent local domain and $\nu_1$, $\nu_2$ are $m_R$-valuations. If the Minkowski equality
$$
e_R(\mathcal I(\nu_1)\mathcal I(\nu_2);R)^{\frac{1}{2}}= e_R(\mathcal I(\nu_1);R)^{\frac{1}{2}}+e_R(\mathcal I(\nu_2);R)^{\frac{1}{2}}
$$
holds then $\nu_1=\nu_2$.
\end{Corollary}

 \begin{proof}
We have by Theorem \ref{PropR11} that $I(\nu_1)_{an}=I(\nu_2)_{bn}$ for all $n$ and some positive, relatively prime integers $a$ and $b$. 

Suppose that $0\ne f\in I(\nu_1)_n$. Then $f^a\in I(\nu_1)_{an}=I(\nu_2)_{bn}$ so that $a\nu_2(f)\ge bn$. If $f^a\in I(\nu_2)_{bn+1}$ then $f^{ab}\in I(\nu_2)_{b(bn+1)}=I(\nu_1)_{a(bn+1)}$ so that $\nu_1(f)>n$. Thus
\begin{equation}\label{eqX21}
\nu_1(f)=n\mbox{ if and only if }\nu_2(f)=\frac{b}{a}n.
\end{equation}
Further, (\ref{eqX21}) holds for every nonzero $f\in {\rm QF}(R)$ since $f$ is a quotient of nonzero elements of $R$.

Now the maps 
 $\nu_1:{\rm QF}(R)\setminus \{0\}\rightarrow \ZZ$ and 
$\nu_2:{\rm QF}(R)\setminus \{0\}\rightarrow \ZZ$ are surjective, so there exists $0\ne f\in {\rm QF}(R)$ such that $\nu_1(f)=1$  and there exists $0\ne g\in {\rm QF}(R)$ such that $\nu_2(g)=1$ which implies that $a=b=1$ since $a,b$ are relatively prime. Thus $\nu_1=\nu_2$.

\end{proof}

\section{Geometry above algebraic  local rings}\label{SecGLR}

\subsection{Intersection products and multiplicity on local rings}

Let $K$ be an algebraic function field over a field $k$. An algebraic local ring of $K$ is a local ring $R$ which is a localization of a finitely generated $k$-algebra and is a domain whose quotient field is $K$. Let $R$ be a $d$-dimensional algebraic normal local ring of $K$. Let ${\rm BirMod}(R)$ be the directed set of blowups 
$\phi:X\rightarrow \mbox{Spec}(R)$  of an $m_R$-primary ideal $I$ of $R$ such that $X$ is normal.

Suppose that $\phi:X\rightarrow \mbox{Spec}(R)$ is in ${\rm BirMod}(R)$.
Let $\{E_1,\ldots,E_t\}$ be the irreducible exceptional divisors of $\phi$.  We define $M^1(X)$ to be the subspace of the real  vector space $E_1\RR+\cdots + E_t\RR$ which is generated by the Cartier divisors. An element of $M^1(X)$ will be called an $\RR$-divisor on $X$. We will say that $D\in M^1(X)$ is a $\QQ$-Cartier divisor if there exists $n\in \ZZ_+$ such that $nD$ is a Cartier divisor.

 We give $M^1(X)$ the Euclidean topology. We first define a natural intersection product $(D_1\cdot D_2\cdot\ldots\cdot D_d)$ on $X$ for $D_1,\ldots,D_d\in M^1(X)$.  The intersection product is a restriction of the one defined in \cite{Kl}. 
We first define the intersection product for Cartier divisors $D_1,\ldots, D_d\in E_1\ZZ+\cdots +E_t\ZZ$. Since this product is multilinear,  it extends naturally to a multilinear product on $M^1(X)^d$.  

There exists a subfield $k_1$ of $K$  such that $k\subset k_1\subset R$ and $R/m_R$ is a finite extension of $k_1$. Thus there exists a projective $k_1$-variety $Y$ and a closed point $q\in Y$ such that $\mathcal O_{Y,q}=R$.  The $m_R$-primary ideal $I$ naturally extends to an ideal sheaf $\mathcal I$ in $\mathcal O_Y$, defined by   
$$
\mathcal I_a=\left\{\begin{array}{ll}
\mathcal O_{Y,a}&\mbox{ if }q\ne a\in Y\\
I&\mbox{ if }a=q.
\end{array}\right.
$$
Let $\Psi:Z\rightarrow Y$ be the projective, birational morphism which is the obtained by blowing up $\mathcal I$. Observe that
base change of this map by $\mathcal O_{Y,q}=R$ gives the original map $\phi:X\rightarrow \mbox{Spec}(R)$.
We can thus view $E_1,\ldots,E_t$ as closed projective subvarieties of the normal variety $Z$.

 Suppose that $F_1,\ldots,F_s$ are Cartier divisors on $Z$ and $\mathcal F$ is a coherent sheaf on $Z$, such that $\dim \mbox{supp }\mathcal F\le s$.
 By \cite{Kl} (surveyed in Chapter 19 of \cite{C5}) we have an intersection product $I(F_1,\ldots, F_s,\mathcal F)$ on $Z$ 
which has the good properties explained in \cite{Kl} and \cite{C5}. The Euler characteristic 
$$
\chi(\mathcal O_Z(n_1F_1+\cdots+n_sF_s)\otimes \mathcal F)=\sum_{i=0}^{\infty}(-1)^ih^i(Z,\mathcal O_Z(n_1F_1+\cdots+n_sF_s)\otimes \mathcal F)
$$
where $h^i(Z,\mathcal G)=\dim_{k_1}H^i(Z,\mathcal G)$ for $\mathcal G$ a coherent sheaf on $Z$, is a polynomial in $n_1,\ldots,n_s$ (\cite{Kl}, \cite[Theorem 19.1]{C5}).  The intersection product $I(F_1,\ldots,F_s,\mathcal F)$ is defined to be the coefficient of $n_1\cdots n_s$ in the Snapper polynomial $\chi(\mathcal O_Z(n_1F_1+\cdots+n_sF_s)\otimes \mathcal F)$. 
We always have that  $I(F_1,\ldots,F_s,\mathcal F)\in \ZZ$.

If $D_1,\ldots,D_s$ are Cartier divisors in $E_1\ZZ+\cdots +E_t\ZZ$, and $\mathcal F$ is a coherent sheaf on $X$ whose support is contained in $\phi^{-1}(m_R)$ (so that $\mathcal F$ naturally extends to a coherent sheaf on $Z$ with the same support) and $\dim\mbox{supp }\mathcal F\le s$, then we define an intersection product
$$
(D_1\cdot\ldots\cdot D_s\cdot \mathcal F)=\frac{1}{[R/m_R:k_1]}I(D_1,\ldots, D_s, \mathcal F)
$$
on $X$. If $W$ is a closed subscheme of $\phi^{-1}(m_R)$, we define 
$$
(D_1\cdot\ldots\cdot D_s\cdot W)=(D_1\cdot\ldots\cdot D_s\cdot\mathcal O_W).
$$
If $s=d$, then we define 
$$
(D_1\cdot\ldots\cdot D_d)=(D_1\cdot\ldots\cdot D_d\cdot X)=
\frac{1}{[R/m_R:k_1]}I(D_1,\ldots, D_s,\mathcal O_Z).
$$
This product is well defined (independent of any choices made in the construction), as follows from the good properties of the intersection product (\cite{Kl}, \cite{C5}). This product naturally extends to a multilinear product on $M^1(X)^d$. 

We will say that a divisor $F=a_1E_1+\cdots+a_tE_t\in M^1(X)$ is effective if $a_i\ge 0$ for all $i$, and anti-effective if $a_i\le 0$ for all $i$. This defines a partial order $\le$ on $M^1(X)$ by $A\le B$ if $B-A$ is effective.  The effective cone ${\rm EF}(X)$ is the closed convex cone in $M^1(X)$  of effective $\RR$-divisors.
The anti-effective cone ${\rm AEF}(X)$ is the closed convex cone in $M^1(X)$ consisting of all anti-effective $\RR$-divisors. 

We will say that an anti-effective divisor $F\in M^1(X)$ is  numerically effective (nef) if 
$$
(F\cdot C)= (F\cdot \mathcal O_C) \ge 0
$$
 for all closed curves $C$ in $\phi^{-1}(m_R)$.  The nef cone $\mbox{Nef}(X)$ is the closed convex cone in $M^1(X)$ of all nef $\RR$-divisors on $X$.

\begin{Lemma} There is an  inclusion of cones ${\rm Nef}(X)\subset {\rm AEF}(X)$.
\end{Lemma}

\begin{proof} Suppose there exists a nef divisor $D\in M^1(X)$ which is not anti-effective. Since $X$ is the blowup of an $m_R$-primary ideal, there exists an anti-effective ample Cartier divisor $A=a_1E_1+\cdots+a_tE_t$, with $a_1,\ldots,a_t<0$. There exists a smallest $\lambda\in \RR$ such that $D+\lambda A$ is anti-effective. Necessarily, $\lambda>0$ and $D+\lambda A$ is nef. Expand $D+\lambda A=\sum b_iE_i$. After possibly reindexing the $E_i$, we have that there exists a number $s$ with $1\le s<t$ such that
$b_1=\cdots =b_s=0$ and $b_{s+1},\ldots,b_t<0$. Now $\phi^{-1}(m_R)$ is connected by Zariski's connectedness theorem (\cite[Section 20]{Z1} or \cite[Corollary III.4.3.2]{G1}).  After reindexing the $E_1,\ldots,E_s$ and the $E_{s+1},\ldots,E_t$, we may assume that $E_s\cap E_{s+1}\ne \emptyset$. Let $C$ be a closed curve on the projective variety $E_s$ which is not contained in $E_{i}$ for $i\ge s+1$ but intersects $E_{s+1}$. Then $((D+\lambda A)\cdot C)<0$, a contradiction.
\end{proof}

We will say that an anti-effective  Cartier divisor $F\in M^1(X)$ is ample on $X$ if there exists an ample Cartier divisor $H$ on $Y$ such that $\Psi^{-1}(H)+F$ is ample on $Z$. This definition is independent of the choice of $Y$ in the construction. We define a divisor $F\in M^1(X)$ to be  ample if $F$ is a formal sum $F=\sum a_iF_i$ where $F_i$ are ample anti-effective Cartier divisors and   $a_i$ are positive real numbers. A divisor $D$ is anti-ample if $-D$ is ample. We define the convex cone
$$
\mbox{Amp}(X)=\{F\in M^1(X)\mid F\mbox{ is ample}\}.
$$

We have that $\mbox{Amp}(X)\subset\mbox{Nef}(X)$, the closure of $\mbox{Amp}(X)$ is $\mbox{Nef}(X)$,
and the interior of $\mbox{Nef}(X)$ is $\mbox{Amp}(X)$, as in \cite{Kl}, \cite[Theorem 1.4.23]{LV1}.

\begin{Remark}\label{Remark2}
If $G\in M^1(X)$, then there exists an effective $\QQ$-divisor $D\in M^1(X)$ such that $G-D\in \mbox{Amp}(X)$.
\end{Remark}

For $F\in M^1(X)$ an effective Cartier divisor, define $I(F)=\Gamma(X,\mathcal O_X(-F))$, an $m_R$-primary ideal in $R$ since $R$ is normal. Let $\pi:Y\rightarrow \mbox{Spec}(k_1)$ be the structure morphism.

\begin{Lemma}\label{Lemma1}  Suppose that $A\in M^1(X)$ is an effective Cartier divisor such that $-A$ is nef. Then
$$
\lim_{m\rightarrow\infty}\frac{\ell_R(R/I(mA))}{m^d}=\frac{-((-A)^d)}{d!}.
$$
\end{Lemma}

\begin{proof}  Let $H$ be an ample Cartier divisor on $Y$ and $L=\Psi^*(H)$. There exists $a\in \ZZ_+$ such that $aL-A$ is nef and big on $Z$.  

We have that $R^1\Psi_*\mathcal O_Z(m(aL-A))\cong \mathcal O_Y(maH)\otimes R^1\Psi_*\mathcal O_Z(-mA)$ 
is a coherent sheaf of $\mathcal O_Y$-modules whose support is $q$ and
\begin{equation}\label{eq6}
H^1(X,\mathcal O_X(-mA))\cong\pi_*(R^1\Psi_*\mathcal O_Z(m(aL-A)))
\end{equation}
as an $R=\mathcal O_{Y,q}$-module.

By \cite[Theorem 6.2]{Fu},
\begin{equation}\label{eq1}
\lim_{m\rightarrow \infty}\frac{h^i(Z,\mathcal O_Z(mG))}{m^d}=0\mbox{ if }i>0
\end{equation}
if $G$ is a nef Cartier divisor on $Z$. 

Now tensor the short exact sequence
$$
0\rightarrow \mathcal O_Z(-mA)\rightarrow \mathcal O_Z\rightarrow \mathcal O_{mA}\rightarrow 0
$$
with $\mathcal O_Z(maL)$ to get a short exact sequence
$$
0\rightarrow \mathcal O_Z(m(aL-A))\rightarrow \mathcal O_Z(maL)\rightarrow \mathcal O_{mA}\otimes\mathcal O_Z(maL)\cong \mathcal O_{mA}\rightarrow 0.
$$
Taking the long exact cohomology sequence, we have that
$$
\lim_{m\rightarrow \infty} \frac{h^i(Z,\mathcal O_{mA})}{m^d}=0
$$
for $i>0$ by (\ref{eq1}), and so

\begin{equation}\label{eq2}
\begin{array}{lll}
\lim_{m\rightarrow\infty}\frac{h^0(Z,\mathcal O_{mA})}{m^d}&=& \lim_{m\rightarrow \infty}\frac{\chi(\mathcal O_{mA})}{m^d}\\
&=& \lim_{m\rightarrow \infty}\frac{\chi(\mathcal O_Z)-\chi(\mathcal O_Z(-mA))}{m^d}\\
&=& \lim_{m\rightarrow\infty}\frac{-\chi(\mathcal O_Z(-mA))}{m^d}\\
&=& \frac{-((-A)^d)}{d!},
\end{array}
\end{equation}
for instance by \cite[Theorem 19.16]{C5}. The end of the cohomology 5 term sequence (forinstance in \cite[Theorem 11.2]{RM}) of the Leray spectral sequence
$$
R^i\pi_*R^j\Psi_*\mathcal O_Z(m(aL-A))\Rightarrow R^{i+j}(\pi\circ \Psi)_*\mathcal O_Z(m(aL-A))
$$
is the exact sequence
\begin{equation}\label{eq3}
%0\rightarrow R^1\pi_*(\Psi_*\mathcal O_Z(m(aL-A)))\rightarrow 
R^1(\pi\circ\Psi)_*\mathcal O_Z(m(aL-A))
\rightarrow \pi_*(R^1\Psi_*\mathcal O_Z(m(aL-A)))\rightarrow R^2\pi_*(\Psi_*\mathcal O_Z(m(aL-A))).
\end{equation}
Now $R^1(\pi\circ\Psi)_*\mathcal O_Z(m(aL-A))=H^1(Z,\mathcal O_Z(m(aL-A))$,
$$
R^2\pi_*(\Psi_*\mathcal O_Z(m(aL-A)))=H^2(Y,\Psi_*\mathcal O_Z(m(aL-A)))=H^2(Y,\mathcal O_Y(maL )\otimes\Psi_*\mathcal O_Z(-mA))    
$$
and $\pi_*(R^1\Psi_*\mathcal O_Z(m(aL-A)))=H^0(Y,R^1\Psi_*\mathcal O_Z(m(aL-A)))$.

Let $\mathcal I_m=\Psi_*\mathcal O_Z(-mA)$. From the  short exact sequences
$$
0\rightarrow \mathcal I_m\otimes\mathcal O_Y(maL)\rightarrow \mathcal O_Y(maL)\rightarrow \mathcal O_Y/\mathcal I_m\rightarrow 0,
$$
we obtain the  exact cohomology sequences
$$
H^1(Y,\mathcal O_Y/\mathcal I_m)\rightarrow H^2(Y,\mathcal I_m\otimes \mathcal O_Z(maL))\rightarrow H^2(Y,\mathcal O_Y(maL)).
$$
Now $H^1(Y,\mathcal O_Y/\mathcal I_m)=0$ since $\mathcal O_Y/\mathcal I_m$ has zero dimensional support and $H^2(Y,\mathcal O_Y(maL))=0$ for $m\gg 0$ since $L$ is ample. Thus
\begin{equation}\label{eq4}
H^2(Y,\mathcal O_Y(maL)\otimes\Psi_*\mathcal O_Z(-mA))=0\mbox{ for }m\gg 0.
\end{equation}
We have
\begin{equation}\label{eq7}
\begin{array}{lll}
\lim_{m\rightarrow\infty}\frac{\ell_R(H^1(X,\mathcal O_X(-mA))}{m^d}&=& 
\lim_{m\rightarrow\infty}\frac{1}{[R/m_R:k_1]}\frac{\dim_{k_1}H^1(X,\mathcal O_X(-mA))}{m^d}\\
&=& \lim_{m\rightarrow\infty}\frac{1}{[R/m_R:k_1]}\frac{h^0(Y,R^1\Psi_*\mathcal O_Z(m(aL-A))}{m^d}\\
&=&0
\end{array}
\end{equation}
by (\ref{eq6}), (\ref{eq3}), (\ref{eq4}) and (\ref{eq1}) with $G=aL-A$ in (\ref{eq1}). We have that $R=H^0(X,\mathcal O_X)$ since $R$ is normal. Now from the exact sequences of $R$-modules
$$
0\rightarrow R/I(mA)\rightarrow H^0(X,\mathcal O_X/\mathcal O_X(-mA))\rightarrow H^1(X,\mathcal O_X(-mA)),
$$
(\ref{eq2}) and (\ref{eq7}) we obtain the formula of the statement of the lemma.
\end{proof}

\begin{Lemma}\label{Lemma3} Suppose that $D_1,\ldots,D_r\in M^1(X)$ are effective Cartier divisors and 
$\mathcal O_X(-D_i)$ is generated by global sections for all $i$. Then for $n_1,\ldots,n_r\in \NN$,
$$
\lim_{m\rightarrow\infty}\frac{\ell_R(R/I(mn_1D_1)\cdots I(mn_rD_r))}{m^d}=
-\frac{((-n_1D_1-\cdots-n_rD_r)^d)}{d!}.
$$
\end{Lemma}
\begin{proof} We have that
$$
I(mn_1D_1)\cdots I(mn_rD_r)\mathcal O_X=\mathcal O_X(-m(n_1D_1+\cdots+n_rD_r))=I(mn_1D_1+\cdots+n_rD_r)\mathcal O_X
$$
since the $\mathcal O_X(-mn_iD_i)$ are generated by global sections.  Thus the integral closure of 
$I(mn_1D_1)\cdots I(mn_rD_r)$ is $I(m(n_1D_1+\cdots+n_rD_r))$ for all $m\in \NN$, and so the $R$-algebra 
$\bigoplus_{m\ge 0}I(m(n_1D_1+\cdots+_rD_r))$ is integral over the $R$-algebra $\bigoplus_{m\ge 0}I(mn_1D_1)\cdots I(mn_rD_r)$. Thus
$$
\begin{array}{lll}
\lim_{m\rightarrow\infty}\frac{\ell_R(R/I(mn_1D_1)\cdots I(mn_rD_r))}{m^d} &=& \lim_{m\rightarrow\infty}
\frac{\ell_R(R/I(mn_1D_1+\cdots+mn_rD_r)}{m^d}\\
&=&-\frac{((-n_1D_1-\cdots-n_rD_r)^d)}{d!}
\end{array}
$$
by Theorem \ref{Theorem13} and Lemma \ref{Lemma1}.
\end{proof}

\subsection{Finite dimensional vector spaces and cones} Suppose that $X\in{\rm BirMod}(R)$. Let $E_1,\ldots,E_r$ be the exceptional components of $X$ for the morphism $X\rightarrow \mbox{Spec}(R)$.
 For $0< p\le d$, we define
 $M^p(X)$ to be the direct product of $M^1(X)$ $p$ times, and we define $M^0(X)=\RR$.
 For $1< p\le d$,  we define
 $L^p(X)$ to be the vector space of $p$-multilinear forms from $M^p(X)$ to $\RR$, and define $L^0(X)=\RR$.

The intersection product gives us 
 $p$-multilinear maps
\begin{equation}\label{eq4*}
M^p(X)\rightarrow L^{d-p}(X)
\end{equation}
 for $0\le p\le d$.
In the special case when $p=0$, the map is just the linear map taking $1$ to the map 
$$
(\mathcal L_1,\ldots,\mathcal L_d)\mapsto (\mathcal L_1\cdot\ldots\cdot \mathcal L_d)=(\mathcal L_1\cdot\ldots\cdot\mathcal L_d\cdot X).
$$
We will denote the image of $(\mathcal L_1,\ldots,\mathcal L_p)$ by $\mathcal L_1\cdot\ldots\cdot \mathcal L_p$.
We will sometimes write 
$$
\mathcal L_1\cdot\ldots\cdot\mathcal L_p(\beta_{p+1},\ldots,\beta_d)
=(\mathcal L_1\cdot\ldots\cdot\mathcal L_p\cdot\beta_{p+1}\cdot\ldots\cdot \beta_d).
$$

We give all  the  vector spaces just defined the Euclidean topology, so that all of the mappings considered above are continuous.

 Let $|\mathcal L|$ be a norm on $M^1(X)$ giving the Euclidean topology. The Euclidean topology on $L^p(X)$ is given by
the norm $||A||$, which is defined on a multilinear form $A\in L^p(X)$ to be the greatest lower bound of all real numbers $c$ such that
$$
|A(x_1,\ldots,x_p)|\le c|x_1|\cdots |x_p|
$$
for $x_1,\ldots,x_p\in M^1(X)$.

Suppose that $V$ is a closed $p$-dimensional  subvariety of some $E_i$  with $1\le p\le d-1$. Define $\sigma_V\in L^p(X)$ by 
$$
\sigma_V(\mathcal L_1,\ldots,\mathcal L_p) = (\mathcal L_1\cdot\ldots\cdot \mathcal L_p\cdot V)
$$
for $\mathcal L_1,\ldots,\mathcal L_p\in M^1(X)$. For $p=d$, define $\sigma_X\in L^d(X)$ by 
$$
\sigma_X(\mathcal L_1,\ldots,\mathcal L_d)=(\mathcal L_1\cdot\ldots\cdot \mathcal L_d)
=(\mathcal L_1\cdot\ldots\cdot \mathcal L_d\cdot X).
$$
The pseudoeffective cone $\mbox{Psef}(L^p(X))$ in $L^p(X)$ is the closure of the cone generated by all such  $\sigma_V$ in $L^p(X)$.  We define
$\mbox{Psef}(L^0(X))$  to be the nonnegative real numbers.

Let $V$ be a  vector space and $C\subset V$ be a pointed (containing the origin) convex cone which is strict
($C\cap(-C)=\{0\}$). Then we have a partial order on $V$ defined by $x\le y$ if $y-x\in C$.

\begin{Lemma}\label{Lemma3.1*}  Suppose that $X\in {\rm BirMod}(R)$ and $1\le p\le d$.
\begin{enumerate}
\item[1)] Suppose that $\alpha\in \mbox{Psef}(L^p(X))$ and $\mathcal L_1,\ldots,\mathcal L_p\in M^1(X)$ are nef. Then
$$
\alpha(\mathcal L_1,\ldots,\mathcal L_p)\ge 0.
$$
\item[2)] $\mbox{Psef}(L^p(X))$ is a strict cone.
\end{enumerate}
\end{Lemma}

The proof of Lemma \ref{Lemma3.1*} is as the proof of \cite[Lemma 3.1]{C4}.

Since $\mbox{Psef}(L^p(X))$ is a strict cone, we  have a partial order on $L^p(X)$, defined by 
$$
\alpha\ge 0\mbox{ if }\alpha\in \mbox{Psef}(L^p(X)).
$$

 We have that $\ge$ is the usual order on $\RR$ since
$L^0(X)=\RR$ and $\mbox{Psef}(L^0(X))$ is the set of nonnegative real numbers.
We also have the partial order on $M^1(X)$ defined by $\alpha\ge 0$ if $\alpha$ is effective.

\begin{Lemma}\label{Remark1} Suppose that $F_1,\ldots,F_p\in M^1(X)$ are such that $F_1$ is anti-effective and  $F_2,\ldots,F_p$ are nef. Then $F_1\cdot\ldots\cdot F_p\le 0$ in $L^{d-p}(X)$.
\end{Lemma}

\begin{proof} We have that $-F_1\in M^1(X)$ is effective. Thus $(-F_1)\cdot F_2\cdot\ldots\cdot F_p\in {\rm Psef}(L^{d-p}(X))$  by Lemma 3.11 \cite{C4}.
\end{proof}

\begin{Lemma}\label{Lemma3.2*}
Suppose that $\beta\in \mbox{Psef}(L^p(X))$. Then  the set
$$
\{\alpha\in \mbox{Psef}(L^p(X))\mid 0\le \alpha\le\beta\}
$$
is compact. 
\end{Lemma}

The proof of Lemma \ref{Lemma3.2*} is the same as the proof of \cite[Lemma 3.2]{C4}.

Suppose that $X,Y\in{\rm BirMod}(R)$  and $f:Y\rightarrow X$ is an $R$-morphism. Then $f$  induces continuous linear maps
$f^*:M^1(X)\rightarrow M^1(Y)$ (from $f^*$ of a Cartier divisor), $f^*:M^p(X)\rightarrow M^p(Y)$ and $f_*:L^p(Y)\rightarrow L^p(X)$. By Proposition I.2.6 \cite{Kl}, for $1\le t\le d$, we have that
\begin{equation}\label{eq5*}
f^*(\mathcal L_1)\cdot\ldots\cdot f^*(\mathcal L_d)=\mathcal L_1\cdot \ldots \cdot \mathcal L_d
\end{equation}
for $\mathcal L_1,\ldots,\mathcal L_d\in M^1(X)$. Thus for $0\le p\le d$ we have commutative diagrams of linear maps
\begin{equation}\label{eq4**}
\begin{array}{rcl}
M^p(Y)&\rightarrow &L^{d-p}(Y) \\
f^*\uparrow&&f_*\downarrow \\
M^p(X)&\rightarrow &L^{d-p}(X).
\end{array}
\end{equation}

For $\alpha\in M^1(X)$, we have that
\begin{equation}\label{eq7*}
f^*(\alpha)\in \mbox{Nef}(Y)\mbox{ if and only if }\alpha\in \mbox{Nef}(X) 
\end{equation}
and
\begin{equation}\label{eq8*}
f^*(\alpha)\mbox{ is effective on $Y$ if and only if $\alpha$ is effective on $X$.}
\end{equation}

\begin{Lemma}\label{Lemma3.3*} Suppose that $X,Y\in {\rm BirMod}(R)$  and $f:Y\rightarrow X$ is an
$R$-morphism. Then $f_*(\mbox{Psef}(L^p(Y)))\subset \mbox{Psef}(L^p(X))$.
\end{Lemma}

The proof of Lemma \ref{Lemma3.3*} is as the proof of \cite[Lemma  3.3]{C4}.

\subsection{Infinite dimensional topological spaces} We have that ${\rm BirMod}(R)$ is a directed set by the $R$-morphisms $Y\rightarrow X$ for $X,Y\in {\rm BirMod}(R)$.     There is at most one $R$-morphism $X\rightarrow Y$ for $X,Y\in {\rm BirMod}(X)$.

The set 
$\{M^p(Y_i)\mid Y_i\in {\rm BirMod}(R)\}$ is a directed system of real vector spaces, where we have a linear mapping 
$f_{ij}^*:M^p(Y_i)\rightarrow M^p(Y_j)$ if the natural birational map $f_{ij}:Y_j\rightarrow Y_i$ is an $R$-morphism. 
We define 
$$
M^p(R)=\lim_{\rightarrow}M^p(Y_i)
$$
with the strong topology  (the direct limit topology, c.f. Appendix 1. Section 1 \cite{D}).  Let $\rho_{Y_i}:M^p(Y_i)\rightarrow M^p(R)$ be the natural mappings. A set $U\subset M^p(R)$ is open if and only if $\rho_{Y_i}^{-1}(U)$ is open in $M^p(Y_i)$ for all $i$.

We have that $M^p(R)$ is a real vector space. As a vector space, $M^p(R)$ is isomorphic to the $p$-fold product $M^1(R)^p$.

We define $\alpha\in M^1(R)$ to be $\QQ$-Cartier (respectively nef or effective) if there exists a representative of $\alpha$ in $M^1(Y)$
which has this  property for some $Y\in {\rm BirMod}(R)$.  We define  $\mbox{Nef}^{\,p}(R)$ to be the  subset of $M^p(R)$ of nef divisors. We define ${\rm EF}^p(R)$ to be the subset of $M^p(R)$ of effective divisors and define $\mbox{AEF}^p(R)$ to be the subset of $M^p(R)$ of anti-efective divisors.  Both of these sets are  convex cones in the vector space $M^p(R)$.

By (\ref{eq7*}) and (\ref{eq8*}), $\{\mbox{Nef}(Y)^p\}$, $\{\mbox{EF}(Y)^p\}$  and $\{\mbox{AEF}(Y)^p\}$
also form directed systems. As sets, we have that
$$
 \mbox{Nef}^{\,p}(R)=\lim_{\rightarrow}(\mbox{Nef}(Y)^p),\,\, {\rm EF}^p(R)=\lim_{\rightarrow}(\mbox{EF}(Y)^p)
 \mbox{ and } \mbox{AEF}^p(R)=\lim_{\rightarrow}(\mbox{AEF}(Y)^p).$$
We give all of these sets  their respective strong topologies. 

Let  $\rho_Y:M^p(Y)\rightarrow M^p(R)$ be the
 induced continuous linear maps
for $Y\in {\rm BirMod}(R)$. We will also denote the induced continuous maps 
$\mbox{Nef}(Y)^p\rightarrow \mbox{Nef}^{\,p}(R)$,  $\mbox{EF}(Y)^p\rightarrow \mbox{EF}^p(R)$ and $\mbox{AEF}(Y)^p\rightarrow \mbox{AEF}^p(R)$  by $\rho_Y$.

The set $\{L^p(Y_i)\}$ is an inverse  system of topological vector spaces, where we have a linear map $(f_{ij})_*:L^p(Y_j)\rightarrow L^p(Y_i)$ if the birational map $f_{ij}:Y_j\rightarrow Y_i$ is a morphism. 
We  define  
$$
L^p(R)=\lim_{\leftarrow}L^p(Y_i),
$$
with the weak topology (the inverse limit topology). Thus the open subsets of $L^p(R)$ are the sets obtained by finite intersections and arbitrary unions of sets $\pi_{Y_i}^{-1}(U)$ where $\pi_{Y_i}:L^p(R)\rightarrow L^p(Y_i)$ is the natural projection and $U$ is open in $L^p(Y_i)$.

In general, good topological properties on a directed system do not extend to the direct limit (c.f. Section 1 of Appendix 2 \cite{D}, especially the remark before 1.8). In particular, we cannot assume that $M^1(R)$ is a {\it topological} vector space. However, good topological properties on an inverse system do extend (c.f. Section 2 of Appendix 2 \cite{D}). In particular, we have the following proposition.

\begin{Proposition}
 $L^p(R)$ is a  Hausdorff real topological vector space which is isomorphic (as a vector space) to the $p$-multilinear forms on $M^1(R)$.
\end{Proposition} 

Let $\pi_Y:L^p(R)\rightarrow L^p(Y)$  be the induced continuous linear maps
for $Y\in {\rm BirMod}(R)$.

The following lemma follows from the universal properties of the inverse limit and the direct limit
(c.f. Theorems 2.5 and 1.5 \cite{D}).

\begin{Lemma}\label{Lemma3.5*} Suppose that $\mathcal F$ is $M^p$ or $\mbox{Nef}^{\,p}$ Then  giving a continuous mapping
$$
\Phi:\mathcal F(R)\rightarrow L^{d-p}(R)
$$
is equivalent to giving continuous maps $\phi_Y:\mathcal F(Y)\rightarrow L^{d-p}(Y)$ for all $Y\in {\rm BirMod}(R)$, such that the diagram
$$
\begin{array}{lll}
\mathcal F(Z)&\stackrel{\phi_Z}{\rightarrow}&L^{d-p}(Z)\\
f^*\uparrow&&\downarrow f_*\\
\mathcal F(Y)&\stackrel{\phi_Y}{\rightarrow}&L^{d-p}(Y)
\end{array}
$$
commutes, whenever $f:Z\rightarrow Y$ is in ${\rm BirMod}(R)$.
\end{Lemma}
In the case when $\mathcal F=M^p$, if the $\phi_Y$ are all multilinear, then $\Phi$ is also multilinear (via the vector space isomorphism
of $M^p(R)$ with $p$-fold product $M^1(R)^p$).

As an application, we have the following useful property.

\begin{Lemma}\label{3.6*} The  intersection product gives us a continuous 
map
$$
\mathcal F(R)\rightarrow L^{d-p}(R)
$$
whenever $\mathcal F$ is $M^p$ or  $\mbox{Nef}^{\,p}$.
The map is multilinear on $M^p(R)$.
\end{Lemma}
We will denote the image of $(\alpha_1,\ldots,\alpha_p)$ by $\alpha_1\cdot\ldots\cdot\alpha_p$. 
For $\beta_{p+1},\ldots,\beta_d\in M^1(R)$, we will often write
$$
\alpha_1\cdot\ldots\cdot\alpha_p(\beta_{p+1},\ldots,\beta_d)=(\alpha_1\cdot\ldots\cdot\alpha_p\cdot\beta_{p+1}\cdot\ldots\cdot\beta_d).
$$

Given $\alpha\in M^1(R)$, there exists $X\in \mbox{BirMod}(R)$ such that $\alpha$ is represented by an element $D$ of $M^1(X)$. If $Y\in \mbox{BirMod}(R)$ and $f:Y\rightarrow X$ is an $R$-morphism, then $\alpha$ is also represented by $f^*(D)\in M^1(Y)$. To simplify notation, we will often regard $\alpha$ as an element of $M^1(X)$ and of $M^1(Y)$, and write $\alpha\in M^1(X)$ and $\alpha\in M^1(Y)$.

\subsection{Pseudoeffective classes in $L^p(R)$}

We define a class $\alpha\in L^p(R)$ to be pseudoeffective if $\pi_Y(\alpha)\in L^p(Y)$ is pseudoeffective for all $Y\in {\rm BirMod}(R)$.

\begin{Lemma}\label{Lemma3.7*}
The set of pseudoeffective classes $\mbox{Psef}(L^p(R))$ in $L^p(R)$
is a  strict closed convex cone in $L^p(R)$. 
\end{Lemma}

The proof of Lemma \ref{Lemma3.7*} is as the proof of \cite[Lemma 3. 7]{C4}.

 By  Lemma \ref{Lemma3.7*} , we can define a partial order $\ge 0$ on $L^p(R)$ by $\alpha\ge 0$ if $\alpha\in \mbox{Psef}(L^p(R))$.
 
 We have that 
$L^0(R)=\RR$ and $\mbox{Psef}(L^0(R))$ is the set of nonnegative real numbers (by the remark before Lemma \ref{Lemma3.1*}), so $\ge$ is the usual order on $\RR$.

\begin{Lemma}\label{Lemma3.8*}
Suppose that $\mathcal L_1,\ldots,\mathcal L_p\in \mbox{Nef}(R)$ and $\alpha\in \mbox{Psef}(L^p(R))$. Then
$$
\alpha(\mathcal L_1,\ldots,\mathcal L_p)\ge 0.
$$
\end{Lemma}

The proof of Lemma \ref{Lemma3.8*} follows from Lemma \ref{Lemma3.1*} as in the proof of \cite[Lemma 3.8]{C4}.

 \begin{Lemma}\label{Lemma3.9*} Suppose that $Y\in {\rm BirMod}(R)$ and $E_1,\ldots,E_r$ are the irreducible exceptional divisors of $Y\rightarrow \mbox{Spec}(R)$. Suppose that $V\subset Y$ is a $p$-dimensional closed subvariety of some $E_i$. Then there exists $\alpha\in \mbox{Psef}(L^p(R))$ such that $\pi_Y(\alpha)=\sigma_V$.
 \end{Lemma}
 
 The proof of Lemma \ref{Lemma3.9*} is as the proof of \cite[Lemma 3.9]{C4}.

The proof of Lemma \ref{Lemma3.10*} below is as the proof of \cite[Lemma 3.10]{C4}.

\begin{Lemma}\label{Lemma3.10*}
Suppose that $\alpha\in \mbox{Psef}(L^p(R))$. Then the set
$$
\{\beta\in L^p(R) \mid  0\le \beta \le \alpha\}
$$
is compact.
\end{Lemma}

\begin{Lemma}\label{Lemma3.11*} 
Suppose that $\alpha_i\in M^1(R)$ for $1\le i\le p$, with $\alpha_1\in {\rm EF}^1(R)$  and $\alpha_i\in {\rm Nef}^1(R)$  for $i\ge 2$. Then $\alpha_1\cdot\ldots\cdot \alpha_p\in {\rm Psef}(L^{d-p}(R))$.
\end{Lemma}

The proof of Lemma \ref{Lemma3.11*} follows from the proof of \cite[Lemma 3.11]{C4}, using Lemma \ref{Remark1}.

\begin{Proposition}\label{Prop3.12*} 
 Suppose that $\alpha_i$ and $\alpha_i'$ for $1\le i\le p$ are nef classes in $M^1(R)$,
and  that $\alpha_i\ge \alpha_i'$ for $i=1,\ldots,p$. Then
$$
\alpha_1\cdot\ldots\cdot\alpha_p\ge \alpha_1'\cdot\ldots\cdot\alpha_p'
$$
in $L^{d-p}(R)$.
\end{Proposition}

The proof of Propositoin \ref{Prop3.12*} is as the proof of \cite[Proposition 3.12]{C4}. 
 
 \section{anti-positive intersection products}\label{SecAPM3}

We continue in this section with the notation introduced in Section \ref{SecGLR}.

A partially ordered set is directed if any  two elements of it can be dominated by a third. A partially ordered set is filtered if any two elements of it dominate a third.

We state  Lemma \ref{Lemma4.1*} below for completeness. A proof can be found in \cite[Lemma 4.1]{C4}.

\begin{Lemma}\label{Lemma4.1*}
Let $V$ be a Hausdorff topological vector space and $K$ a strict closed convex cone in $V$ with associated partial order relation $\le$.
Then any nonempty subset $S$ of $V$ which is directed with respect to $\le$ and  is contained in a compact subset of $V$
has a least upper bound with respect to $\le$ in $V$.
\end{Lemma}

\begin{Lemma}\label{Lemma4.2*} Suppose that $\alpha\in M^1(R)$ is anti-effective. Then the set $\mathcal D(\alpha)$ of effective $\QQ$-divisors $D$ in $ M^1(R)$ such that $\alpha-D$ is nef is nonempty and filtered.
\end{Lemma}

The proof of Lemma \ref{Lemma4.2*}, using Remark \ref{Remark2},  is as the proof of \cite[Lemma 4.2]{C4}.

The following proposition generalizes \cite[Proposition 4.3]{C4}.

\begin{Proposition}\label{Prop4.3*} Suppose that $\alpha_1,\ldots,\alpha_p\in M^1(R)$ are anti-effective. Let
$$
\begin{array}{lll}
S&=&\{(\alpha_1-D_1)\cdot\ldots\cdot (\alpha_p-D_p)\in L^{d-p}(R)\mbox{ such that }\\
&&D_1,\ldots,D_p\in M^1(R)\mbox{ are effective $\QQ$-divisors and $\alpha_i-D_i$ are nef for $1\le i\le p$}\}.
\end{array}
$$
Then 
\begin{enumerate}
\item[1)] $S$ is nonempty.
\item[2)] $S$ is a directed set with respect to the partial order $\le$ on $L^{d-p}(R)$.
\item[3)] $S$ has a (unique) least upper bound with respect to $\le$ in $L^{d-p}(R)$.
\end{enumerate}
\end{Proposition}

\begin{proof} There exists $\phi:X\rightarrow \mbox{Spec}(R)$ in $\mbox{BirMod}(R)$ such that $\alpha_1,\ldots,\alpha_p\in M^1(X)$. Since $X$ is the blowup of an $m_R$-primary ideal, there exists an effective $\QQ$-divisor $\omega$ in $M^1(R)$ such that $-\omega$ is ample on $X$ and $\alpha_i-\omega$ is nef for all $i$. Suppose  $D_i\in M^1(R)$ are effective $\QQ$-divisors such that 
$\alpha_i-D_i$ are nef for all $i$. Lemma \ref{Lemma4.2*} implies there exist effective $\QQ$-divisors 
$D_i^*\in M^1(R)$ such that for all $i$, $\alpha_i-D_i$ are nef, $D_i^*\le D_i$, $D_i^*\le \omega$ and $\alpha_i-D_i^*$ are nef. Thus $\alpha_i-\omega\le \alpha_i-D_i^*\le 0$ and $\alpha_i-D_i\le \alpha_i-D_i^*$. Proposition \ref{Prop3.12*} implies
$$
(\alpha_1-\omega)\cdot(\alpha_2-\omega)\cdot\ldots\cdot(\alpha_p-\omega)
\le (\alpha_1-D_1^*)\cdot(\alpha_2-D_2^*)\cdot\ldots\cdot(\alpha_p-D_p^*)\le 0.
$$
Thus $\gamma\in L^{d-p}(R)$ is an upper bound for $S$ if and only if $\gamma$ is an upper bound for $S\cap Z$ where
$$
Z=\{x\in L^{d-p}(R)\mid (\alpha_1-\omega)\cdot\ldots\cdot(\alpha_p-\omega)\le x\le 0\}.
$$
The set $S\cap Z$ is nonempty since 
$(\alpha_1-\omega)\cdot\ldots\cdot(\alpha_p-\omega)\in S\cap Z$.
 The set $S\cap Z$ is directed since $S$ is and since whenever $\beta_1,\ldots,\beta_p\in M^1(R)$ are anti-effective and
nef, $\beta_1\cdots\ldots\cdot\beta_p\le 0$ (by Lemma \ref{Remark1}).  The set $Z$ is compact by Lemma \ref{Lemma3.10*}. Thus by Lemma \ref{Lemma4.1*}, $S\cap Z$ has a least upper bound with respect to $\le$ in $L^{d-p}(R)$. 
\end{proof}

The following definition is well defined by Proposition \ref{Prop4.3*}. Definition \ref{Def4.4*} gives a local version  of the definition  \cite[Definition 4.4]{C4} of the positive intersection product on a proper variety.

\begin{Definition}\label{Def4.4*} Suppose that $\alpha_1,\ldots,\alpha_p\in M^1(R)$ are anti-effective. Their anti-positive intersection product $\langle\alpha_1\cdot\ldots\cdot \alpha_p\rangle\in L^{d-p}(R)$ is defined to be the least upper bound of the set of classes $(\alpha_1-D_1)\cdot\ldots\cdot (\alpha_p-D_p)\in L^{d-p}(R)$ where $D_i\in M^1(R)$ are effective $\QQ$-Cartier divisors in $M^1(R)$ such that $\alpha_i-D_i$ are nef. 
\end{Definition}

The proof of the  following proposition is as the proof of Proposition 4.7 \cite{C4}.

\begin{Proposition}\label{Prop4.7*} The map $\mbox{AEF}^p(R)\rightarrow L^{d-p}(R)$ defined by 
$(\alpha_1,\ldots,\alpha_p)\mapsto \langle\alpha_1,\cdot,\ldots, \alpha_p\rangle$ is continuous.
\end{Proposition}

\section{Mixed multiplicities and anti-positive intersection products}\label{SecAPM1}

We continue in this section with the notation of Sections \ref{SecGLR} and \ref{SecAPM3}.
In this section, suppose that $\alpha_1,\ldots \alpha_r\in  M^1(R)$ are effective Cartier divisors. For $n_1,\ldots,n_r\in \NN$, define
$$
F(n_1,\ldots,n_r)=\lim_{m\rightarrow\infty}\frac{\ell_R(R/I(mn_1\alpha_1)\cdots I(mn_r\alpha_r))}{m^d}.
$$
We have that $F(n_1,\ldots,n_r)$ is a homogeneous polynomial of degree $d$ by  \cite[Theorem 6.6]{CSS}.

We now describe a construction that we will use in this section.  Let $X\in \mbox{BirMod}(R)$ be such that 
$\alpha_1,\ldots,\alpha_r\in M^1(X)$. For $s\in \ZZ_+$, let 
\begin{equation}\label{eq13}
Y[s]\rightarrow X
\end{equation}
 be in $\mbox{BirMod}(X)$ and let $\pi_s:Y_s\rightarrow Y[s]$ be the normalization of the blowup of 
 $$
 I(s\alpha_1)\cdots I(s\alpha_r)\mathcal O_{Y[s]}.
 $$
Let $\psi_s:Y_s\rightarrow \mbox{Spec}(R)$ be the induced morphism. Define  effective Cartier divisors $F_{s,i}$ on $Y_s$ by 
$$
I(s\alpha_i)\mathcal O_{Y_s}=\mathcal O_{Y_s}(-F_{s,i})\subset \mathcal O_{Y_s}(\pi_s^*(-s\alpha_i)).
$$
Let $D_{s,i}=F_{s,i}-\pi_s^*(s\alpha_i)$, which we will write as $F_{s,i}-s\alpha_i$. Then $D_{s,i}$ is an effective Cartier divisor on $Y_s$ and $-\alpha_i-\frac{1}{s}D_{s,i}=-\frac{1}{s}F_{s,i}$ is anti-effective and nef. We have that 

\begin{equation}\label{eq12}
\begin{array}{l}
I(s\alpha_1)^{mn_1}\cdots I(s\alpha_r)^{mn_r}\subset I(mn_1F_{s,1})\cdots I(mn_r F_{s,r})\subset I(msn_1\alpha_1)\cdots I(msn_r\alpha_r)\\
\mbox{ for all }m,n_1,\ldots,n_r\in\NN.
\end{array}
\end{equation}
For $n_1,\ldots,n_r\in \NN$, define
$$
H_s(n_1,\ldots,n_r)=\lim_{m\rightarrow\infty}\frac{\ell_R(R/I(mn_1F_{s,1})\cdots I(mn_r F_{s,r}))}{s^dm^d}.
$$
We have that $H_s(n_1,\ldots,n_r)$ is a homogeneous polynomial of degree $d$ in $n_1,\ldots,n_r$ by Theorem  \cite[Theorem 6.6]{CSS}.

Expand  the polynomials 
$$
H_s(n_1,\ldots,n_r)=\sum b_{i_1,\ldots,i_r}(s)n_1^{i_1}\cdots n_r^{i_r}
$$
and
$$
F(n_1,\ldots,n_r)=\sum b_{i_1,\ldots,i_r}n_1^{i_1}\cdots n_r^{i_r}
$$
with $b_{i_1,\ldots,i_r}(s),b_{i_1,\ldots,i_r}\in \RR$.

\begin{Proposition}\label{Prop2} For all $n_1,\ldots,n_r\in\NN$,
$$
\lim_{s\rightarrow \infty}H_s(n_1,\ldots,n_r)=F(n_1,\ldots,n_r)
$$
and for all $i_1,\ldots,i_r$,
\begin{equation}\label{eq177}
\lim_{s\rightarrow \infty}b_{i_1,\ldots,i_r}(s)=b_{i_1,\ldots,i_r}.
\end{equation}
\end{Proposition}

\begin{proof} For $s\in \ZZ_+$, let $\{I_s(j)_i\}$ be the $s$-th truncated filtration of 
$\{I(j)_i\}$ where $I(j)_i=I(i\alpha_j)$ is
defined in \cite[Definition 4.1]{CSS}. That is, $I_s(j)_i=I(i\alpha_j)$ if $i\le s$ and if $i>s$, then $I_s(j)_i=\sum I_s(j)_aI_s(j)_b$ where the sum is over all $a,b>0$ such that $a+b=i$.
Let 
$$
F_s(n_1,\ldots,n_r)=\lim_{m\rightarrow \infty}\frac{\ell_R(R/I_s(1)_{mn_1}\cdots I_s(r)_{mn_r})}{m^d}
$$
for $n_1,\ldots,n_r\in \NN$.  Now there exists $m(s)\in \ZZ_+$ such that 
$$
I_s(1)_{smn_1}\cdots I_s(r)_{smn_r}=I(s\alpha_1)^{mn_1}\cdots I(s\alpha_r)^{mn_r}
$$
for $m\ge m(s)$. By (\ref{eq12}), we have
$$
F_s(n_1,\ldots,n_r)=\frac{F_s(sn_1,\ldots,sn_r)}{s^d}\ge H_s(n_1,\ldots,n_r)\ge\frac{F(sn_1,\ldots,sn_r)}{s^d}=F(n_1,\ldots,n_r)
$$
for all $n_1,\ldots,n_r\in\NN$. By \cite[Proposition  4.3]{CSS}, for all $n_1,\ldots,n_r\in \ZZ_+$,
$$
\lim_{s\rightarrow\infty}F_s(n_1,\ldots,n_r)=F(n_1,\ldots,n_r).
$$
Thus for all $n_1,\ldots,n_r\in \ZZ_+$,
\begin{equation}\label{eq11}
\lim_{s\rightarrow
\infty}H_s(n_1,\ldots,n_s)=F(n_1,\ldots,n_r).
\end{equation}
By \cite[Lemma 3.2]{CSS} and (\ref{eq11}), we have that 
$$
\lim_{s\rightarrow \infty}b_{i_1,\ldots,i_r}(s)=b_{i_1,\ldots,i_r}
$$
for all $i_1,\ldots,i_r$. Thus
$$
\lim_{s\rightarrow\infty}H_s(n_1,\ldots,n_r)=F(n_1,\ldots,n_r)
$$
for all $n_1,\ldots,n_r\in\NN$.
\end{proof}

\begin{Theorem}\label{Theorem1}
The coefficients of $F(n_1,\ldots,n_r)$ are
 $$
b_{i_1,\ldots,i_r}=\frac{-1}{i_1!\cdots i_r!}\langle(-\alpha_1)^{i_1}\cdot\ldots\cdot (-\alpha_r)^{i_r}\rangle
$$
for all $i_1,\ldots,i_r$.
\end{Theorem}

\begin{proof} For $s\in \ZZ_+$, let $\epsilon_s=\frac{1}{2^s}$. There exist effective $\QQ$-Cartier divisors $D_1(s),\ldots, D_r(s)\in M^1(R)$ such that $-\alpha_1-D_1(s),\ldots,-\alpha_r-D_r(s)$ are nef and 
$((-\alpha_1-D_1(s))^{n_1}\cdot\ldots\cdot (-\alpha_r-D_r(s))^{n_r})$ is within $\epsilon_s$ of $\langle(-\alpha_1)^{n_1}\cdot\ldots\cdot(-\alpha_r)^{n_r}\rangle$ for  all $n_1,\ldots,n_r\in \ZZ_+$ with $n_1+\cdots+n_r=d$.
Let $Y(s)\rightarrow X\in \mbox{BirMod}(R)$ be such that $\alpha_1,\ldots,\alpha_r,D_1(s),\ldots,D_r(s)\in M^1(Y(s))$. Let $A_s$ be effective and anti-ample on $Y(s)$. Then by Proposition \ref{Prop4.7*}, for $t>0$ sufficiently small, each product 
$((-\alpha_1-D_1(s)-tA_s)^{n_1}\cdot\ldots\cdot (-\alpha_r-D_r(s)-tA_s)^{n_r})$ is within $\epsilon_s$ of $\langle(-\alpha_1)^{n_1}\cdots\ldots\cdot (-\alpha_r)^{n_r}\rangle$ for all $n_1,\ldots,n_r\in \ZZ_+$ with $n_1+\cdots+n_r=d$.
Replacing $D_i(s)$ with $D_i(s)+tA_s$ for such a small rational $t$, we may assume that $-\alpha_i-D_i(s)$ are ample for all $i$.

There exist $m_i\in \ZZ_+$ for $i\in \ZZ_+$ such that $m_1<m_2<\cdots$, the $m_s\alpha_i$ are effective Cartier divisors on $Y(s)$,   $m_sD_s(s)$ is an effective  Cartier divisor on $Y(s)$ and $\mathcal O_{Y(s)}(-m_s\alpha_i-m_sD_i(s))$ is very ample on $Y(s)$ for all $s$ and $1\le i\le r$.
In (\ref{eq13}), let $Y[m_s]=Y(s)$ for $s\in \ZZ_+$ and $Y[t]=X$ for $t\not\in \{m_1,m_2,\ldots\}$.

With the notation introduced after (\ref{eq13}), let $F_{m_s,i}$ be the Cartier divisor on $Y_{m_s}$ defined by $\mathcal O_{Y_{m_s}}(-F_{m_s,i})=I(m_s\alpha_i)\mathcal O_{Y_{m_s}}$. We  have that 
$$
I(m_s(\alpha_i+D_i(s))=\Gamma(Y(s),\mathcal O_{Y(s)}(-m_s\alpha_i-m_sD_i(s)))\subset
\Gamma(Y(s),\mathcal O_{Y(s)}(-m_s\alpha_i))=I(m_s\alpha_i).
$$
Since $-m_s\alpha_i-m_sD_i(s)$ is very ample on $Y(s)$,
$$
\mathcal O_{Y(s)}(-m_s\alpha_i-m_sD_i(s))=I(-m_s\alpha_i -m_sD_i(s))\mathcal O_{Y(s)}\subset I(m_s\alpha_i)\mathcal O_{Y(s)}.
$$
Thus 
$$
\mathcal O_{Y_{m_s}}(-m_s\alpha_i-m_sD_i(s))\subset I(m_s\alpha_i)\mathcal O_{Y_{m_s}}=\mathcal O_{Y_{m_s}}(-F_{m_s,i})\subset \mathcal O_{Y_{m_s}}(-m_s\alpha_i)
$$
for all $i,s$. Thus
$$
-\alpha_i-D_i(s)\le -\frac{F_{m_s,i}}{m_s}\le -\alpha_i.
$$
Now $\frac{-F_{m_i,s}}{m_s}$ is nef and 
$$
\frac{-F_{m_s,i}}{m_s}=-\alpha_i-E_{m_s,i}
$$
where $E_{m_s,i}$ is an effective $\QQ$-Cartier divisor. We have that
$$
\begin{array}{lll}
((-\alpha_1-D_1(s))^{n_1}\cdots\ldots\cdot (\alpha_r-D_r(s))^{n_r})&\le&
\left(\left(\frac{-F_{m_s,1}}{m_s}\right)^{n_1}\cdot\ldots\cdot\left(\frac{-F_{m_s,r}}{m_s}\right)^{n_r}\right)\\
&\le& \langle(-\alpha_1)^{n_1}\cdot\ldots\cdot (-\alpha_r)^{n_r}\rangle
\end{array}
$$
for all $s$ and $n_1,\ldots,n_r\in \NN$ with $n_1+\cdots+n_r=d$. The first inequality is  by Proposition \ref{Prop3.12*} and the second inequality is by Definition \ref{Def4.4*}.
Thus
\begin{equation}\label{eq14}
\left(\left(\frac{-F_{m_s,1}}{m_s}\right)^{n_1}\cdot\ldots\cdot \left(\frac{-F_{m_s,r}}{m_s}\right)^{n_r}\right)
\mbox{ is within $\epsilon_s$ of }\langle(-\alpha_1)^{n_1}\cdot\ldots\cdot (-\alpha_r)^{n_r}\rangle
\end{equation}
for all $n_1,\ldots,n_r\in \NN$ with $n_1+\cdots+n_r=d$.
\begin{equation}\label{eq15}
\begin{array}{l}
\mbox{Given $\epsilon>0$, for $s\gg 0$, the coefficients $b_{i_1,\ldots,i_r}(m_s)$ of $H_{m_s}(n_1,\ldots,n_r)$}\\
 \mbox{are within $\epsilon$ of the coefficients $b_{i_1,\ldots,i_r}$ of $F(n_1,\ldots,n_r)$}
 \end{array}
\end{equation}
by Proposition \ref{Prop2} and
\begin{equation}\label{eq16}
\begin{array}{l}
\frac{1}{m_s^d}((-F_{m_s,1}))^{i_1}\cdot\ldots\cdot (-F_{m_s,r})^{i_r})\mbox{ is within $\epsilon$ of }
\langle(-\alpha_1)^{i_1}\cdot\ldots\cdot(-\alpha_r)^{i_r}\rangle\\
\mbox{ for all }i_1,\ldots,i_r\in \NN\mbox{ with $i_1+\cdots+i_r=d$}
\end{array}
\end{equation}
 by (\ref{eq14}).  Now 
 \begin{equation}\label{eq17}
 \begin{array}{lll}
 H_{m_s}(n_1,\ldots,n_r)&=&\frac{1}{m_s^d}\left(\lim_{m\rightarrow\infty}\frac{\ell_R(R/I(mn_1F_{m_s,1})\cdots I(mn_rF_{m_s,r}))}{m^d}\right)\\
 &=&\frac{-1}{m_s^dd!}((-n_1F_{m_s,1}-\cdots-n_rF_{m_s,r})^d)
 \end{array}
 \end{equation}
 by Lemma \ref{Lemma3}, since $F_{m_s,1},\ldots, F_{m_s,r}$ are effective Cartier divisors and $\mathcal O_{Y_{m_s}}(-F_{m_s,i})$ are generated by global sections for all $i$. Then expanding the last line of (\ref{eq17}) by the multinomial theorem, we obtain
 $$
 b_{i_1,\ldots,i_r}(m_s)=\frac{-1}{m_s^di_1!\cdots i_r!}((-F_{m_s,1})^{i_1}\cdot \ldots \cdot (-F_{m_s,r})^{i_r})
 $$
 for all $i_1,\ldots, i_r\in \NN$ with $i_1+\cdots+i_r=d$. By (\ref{eq15}) and (\ref{eq16}), we have that
 $$
 b_{i_1,\ldots,i_r}=\frac{-1}{i_1!\cdots i_r!}\langle(-\alpha_1)^{i_1}\cdot\ldots\cdot (-\alpha_r)^{i_r}\rangle
 $$
 for all $i_1,\ldots,i_r$.

\end{proof}

The mixed mutiplicities $e_R(\mathcal I(1)^{[d_1]},\ldots, \mathcal I(r)^{[d_r]};R)$ of the filtrations $\mathcal I(1),\ldots,\mathcal I(r)$ of $m_R$-primary ideals are defined in \cite{CSS} from the coefficients $b_{d_1,\ldots,d_r}$   of $F(n_1,\ldots,n_r)$ by defining 
$$
b_{d_1,\ldots,d_r}=\frac{1}{d_1!\cdots d_r!}e_R(\mathcal I(1)^{[d_1]},\ldots, \mathcal I(r)^{[d_r]};R).
$$

The following theorem follows immediately from Theorem \ref{Theorem1}.

\begin{Theorem}\label{TheoremA} Let $R$ be a normal algebraic local ring, $\alpha_1,\ldots,\alpha_r\in M^1(R)$ be effective Cartier divisors and let $\mathcal I(j)$ be the filtration 
$\mathcal I(j)=\{I(n\alpha_j)\}$ for $1\le j\le r$.

Then the mixed multiplicities 
$$
e_R(\mathcal I(1)^{[d_1]},\ldots,\mathcal I(r)^{[d_r]};R)
=-\langle(-\alpha_1)^{d_1}\cdot\ldots\cdot (-\alpha_r)^{d_r}\rangle
$$
for $d_1,\ldots,d_r\in \NN$ with $d_1+\cdots+d_r=d$ 
are the negatives of the anti-positive intersection products of $-\alpha_1,\ldots,-\alpha_r$.
\end{Theorem}

From the case $r=1$ of Theorem \ref{TheoremA}, we obtain the statement that
$$
e_R(\mathcal I;R)=\langle(-\alpha)^d\rangle
$$
if $\alpha\in M^1(R)$ is an effective  Cartier divisor and $\mathcal I=\{I(m\alpha)\}$.

%\begin{Corollary} Suppose that $R$ is a characteristic zero normal algebraic local domain of dimension $d$ and $\nu_1,\ldots,\nu_r$ are divisorial valuations of the quotient field of $R$ which dominate $R$. Let $\mathcal(\nu_j)=\{I(\nu_j)_n\}$ where $I(\nu_j)_m=\{f\in R\mid \nu_i(f)\ge n\}$.  Let $X\rightarrow \mbox{Spec}(R)$ be a resolution of singularities such that the centers $E_1,\ldots,E_r$ of $\nu_1,\ldots,\nu_r$ on $X$ are codimension 1 subvarieties of $X$. Then 
%$$
%e_R(\mathcal I(\nu_1)^{[d_1]},\ldots,\mathcal I(\nu_r)^{[d_r]};R)=-<(-E_1)^{d_1}\cdot\ldots\cdot (-E_r)^{d_r}>
%$$
%for all $d_1,\ldots,d_r\in\NN$ with $d_1+\cdots+d_r=d$.
%\end{Corollary}

%\begin{proof} Theorem \ref{TheoremA} is applicable since 
%$E_1,\ldots,E_r$ are Cartier divisors on $X$ as $X$ is nonsingular. 
%\end{proof}

\begin{Theorem}\label{Theorem10} Suppose that $R$ is a $d$-dimensional  algebraic local domain, and $\mathcal I(j)=\{I(mD(j))\}$ are divisorial filtrations of $R$ for $1\le j\le r$. 
Then the mixed multiplicities 
$$
e_R(\mathcal I(1)^{[d_1]},\ldots,\mathcal I(r)^{[d_r]})
=\sum_{i=1}^t-[S/m_i:R/m_R]\langle(-D(1)_i)^{d_1}\cdot\ldots\cdot (-D(r)_i)^{d_r}\rangle
$$
for $d_1,\ldots, d_r\in \NN$ with $d_1+\cdots+d_r=d$.
\end{Theorem}

\begin{proof} We use the notation introduced before the statement of Lemma \ref{LemmaR1}.
From Lemma \ref{LemmaR1} and (\ref{eqR15}), we have that
$$
\begin{array}{l}
\lim_{n\rightarrow\infty}\frac{\ell_R(R/I(nn_1D(1)\cdots I(nn_rD(r)))}{n^d}\\
=\sum_{i=1}^t[S/m_i:R/m_R]\left(\lim_{n\rightarrow\infty}\frac{\ell_{S_{m_i}}(S_{m_i}/J(nn_1D(1)_i)\cdots J(nn_rD(r)_i))}{n^d}\right).
\end{array}
$$
The theorem now follows from Theorem \ref{TheoremA}.
\end{proof}

The following theorem follows from Theorem \ref{TheoremA} and \cite[Theorem 1.2]{CSS}. It    shows that the  Minkowski inequalities hold for the absolute values of the anti-positive intersection products.

\begin{Theorem}\label{TheoremB}(Minkowski Inequalities)  Let assumptions be as in Theorem \ref{TheoremA}, with $r=2$. Then 
\begin{enumerate}
\item[1)] $(\langle(-\alpha_1)^i,(-\alpha_2)^{d-i}\rangle)^2\le \langle(-\alpha_1)^{i+1},(-\alpha_2)^{d-i-1}\rangle \langle(-\alpha_1)^{i-1},(-\alpha_2)^{d-i+1}\rangle$
for $1\le i\le d-1$.
\item[2)]  For $0\le i\le d$, 
$$
\langle(-\alpha_1)^i,(-\alpha_2)^{d-i}\rangle\langle(-\alpha_1)^{d-i},(-\alpha_2)^i\rangle
\le \langle(-\alpha_1)^d\rangle\langle(-\alpha_2)^d\rangle,
$$
\item[3)] For $0\le i\le d$, $(-\langle(-\alpha_1)^{d-i},(-\alpha_2)^i\rangle)^d\le (-\langle(-\alpha_1)^d\rangle)^{d-i}(-\langle(-\alpha_2)^d\rangle)^i$ and
\item[4)]  $(-\langle(-\alpha_1-\alpha_2)^d\rangle)^{\frac{1}{d}}\le (-\langle(-\alpha_1)^d\rangle)^{\frac{1}{d}}+(-\langle(-\alpha_2)^d\rangle)^{\frac{1}{d}}$.
\end{enumerate}
\end{Theorem}

We mention a version of the Minkowski inequalities in terms of positive intersection numbers for pseudo effective divisors on a projective variety.

\begin{Theorem} (Minkowski Inequalities) Suppose that $X$ is a complete algebraic variety of dimension $d$ over a field $k$ and $\mathcal L_1$ and $\mathcal L_2$ are pseudo effective Cartier divisors on $X$. Then
\begin{enumerate}
\item[1)] $(\langle\mathcal L_1^i,\mathcal L_2^{d-i}\rangle)^2\ge \langle\mathcal L_1^{i+1},\mathcal L_2^{d-i-1}\rangle
\langle\mathcal L_1^{i-1},\mathcal L_2^{d-i+1}\rangle$ for $1\le i\le d-1.$
\item[2)]  $\langle\mathcal L_1^i,\mathcal L_2^{d-i}\rangle \langle\mathcal L_1^{d-i},\mathcal L_2^i>\rangle\ge
\langle\mathcal L_1^d\rangle\langle\mathcal L_2^d\rangle$ for $1\le i\le d-1$.
\item[3)] $(\langle\mathcal L_1^{d-i},\mathcal L_2^i\rangle)^d\ge (\langle\mathcal L_1^d\rangle)^{d-i}(\langle\mathcal L_2^d\rangle)^i$ for $0\le i\le d$.
\item[4)] $(\langle(\mathcal L_1\otimes\mathcal L_2)^d\rangle)^{\frac{1}{d}}\ge (\langle\mathcal L_1^d\rangle)^{\frac{1}{d}}+(\langle\mathcal L_2^d\rangle)^{\frac{1}{d}}$.
\end{enumerate}
\end{Theorem}

\begin{proof} Statements 1) - 3)  follow from the inequality of Theorem 6.6 \cite{C4}. Statement 4) follows from 3) and \cite[Lemma 4.13]{C4}, which establishes the super additivity of the positive intersection  product.
\end{proof}

\section*{Appendix:  A proof of Theorem \ref{Theorem13} }\label{SecApp} In this appendix we give a proof of Theorem \ref{Theorem13}. We fix a potentially confusing index error in the proof in \cite{CSS}.

 Step 1). We first observe that if $I'\subset I$ are $m_R$-primary ideals and $\bigoplus_{n\ge 0}I^n$ is integral over $\bigoplus_{n\ge 0}(I')^n$, then, by \cite[Theorem 8.2.1, Corollary 1.2.5 and Proposition 11.2.1]{HS},   $e_R(I;R)=e_R(I';R)$.

Step 2).  Suppose $\mathcal I=\{I_i\}$ and $\mathcal I'=\{I'_i\}$ are Noetherian filtrations of $R$ by $m_R$-primary ideals and $\mathcal I'\subset \mathcal I$. Suppose $b\in \ZZ_+$. 
Define $\mathcal I^{(b)}=\{I^{(b)}_i\}$ where $I^{(b)}_i=I_{bi}$ and $(\mathcal I')^{(b)}=\{(I')^{(b)}_i\}$ where $(I')^{(b)}_i=(I')_{bi}$.
Then from \cite[Lemma 3.3]{CSS}  we deduce that
$$
e_R(\mathcal I;R)=e_R(\mathcal I';R)\mbox{ if and only if }e_R(\mathcal I^{(b)};R)=e_R((\mathcal I')^{(b)};R).
$$

Step 3). Suppose $\mathcal I'\subset \mathcal I$ are filtrations of $R$ by $m_R$-primary ideals. Suppose $a\in \ZZ_+$. Let $\mathcal I_a=\{I_{a,n}\}$ be the $a$-th truncated filtration of $\mathcal I$ defined in \cite[Definition 4.1]{CSS}. Then there exists $\overline a\in \ZZ$ such that every element of $\bigoplus_{n\ge 0}I_{a,n}$ (considered as a subring of $\bigoplus_{n\ge 0}I_n$) is integral over $\bigoplus_{n\ge 0}I'_{\overline a,n}$, where $\mathcal I'_{\overline a}=\{I'_{\overline a,i}\}$ is the $\overline a$-th truncated filtration of $\mathcal I'$  defined in \cite[Definition 4.1]{CSS}s.
 
 Define a Noetherian filtration $\mathcal A_a=\{A_{a,i}\}$ of $R$ by $m_R$-primary ideals 
 by 
 $$
 A_{a,i}=\sum_{\alpha+\beta=i}I_{a,\alpha}I'_{\overline a,\beta}.
 $$
 Recall that $I_{a,0}=I'_{\overline a,0}=R$. We restrict to $\alpha,\beta\ge 0$ in the sum.    Thus we have inclusions of graded rings $\bigoplus_{n\ge 0}I'_{\overline a,n}\subset \bigoplus_{n\ge 0}A_{a,n}$ and $\bigoplus_{n\ge 0}A_{a,n}$ is finite over $\bigoplus_{n\ge 0}I'_{\overline a,n}$. By Steps 2) and 1),
$$
e_R(\mathcal I'_{\overline a};R)=e_R(\mathcal A_a;R).
$$
By \cite[Proposition 4.3]{CSS},
$$
\lim_{a\rightarrow \infty}e_R(\mathcal I'_{\overline a};R)=e_R(\mathcal I';R)
$$
and thus 
$$
\lim_{a\rightarrow \infty}e_R(\mathcal A_a;R)=e_R(\mathcal I';R).
$$

Step 4) Let notation be as in the proof of \cite[Proposition 4.3]{CSS}, but taking $J_i=I_i$ and $J(a)_i=I_{a,i}$. Define 
$$
\begin{array}{lll}
\Gamma(\mathcal A_a)^{(t)}&=&\{(m_1,\ldots,m_d,i)\in \NN^{d+1}\mid
\dim_kA_{a,i}\cap K_{m_1\lambda_1+\cdots+m_d\lambda_d}/A_{a,i}\cap K^+_{m_1\lambda_1+\cdots+m_d\lambda_d}\ge t\\
&&\mbox{ and }m_1+\cdots+m_d\le \beta i\}.
\end{array}
$$
Now $\Gamma(a)^{(t)}\subset \Gamma(\mathcal A_a)^{(t)}\subset \Gamma^{(t)}$ for all $t$, so
$$
\Delta(\Gamma(a)^{(t)})\subset \Delta(\Gamma(\mathcal A_a)^{(t)})\subset \Delta(\Gamma^{(t)})
$$
for all $a$. By equation (14) \cite{CSS},
$$
\lim_{a\rightarrow \infty} {\rm Vol}(\Delta(\Gamma(a)^{(t)}))={\rm Vol}(\Delta(\Gamma^{(t)})),
$$
and so 
$$
\lim_{a\rightarrow \infty} {\rm Vol}(\Delta(\Gamma(\mathcal A_a)^{(t)}))={\rm Vol}(\Delta(\Gamma^{(t)})).
$$
Thus
$$
\lim_{a\rightarrow \infty}e_R(\mathcal A_a;R)=e_R(\mathcal I;R)
$$
by (12) of the proof of \cite[Proposition 4.3]{CSS}  applied to $\mathcal A_a$. 

Step 5). We have that $e_R(\mathcal I;R)=e_R(\mathcal I';R)$ by Steps 3) and 4). Now $e_R(\mathcal I;M)=e_R(\mathcal I';M)$ by \cite[Theorem 6.8]{CSS}(with $r=1$).


\begin{thebibliography}{1000000000}
\bibitem{Ba} L. Badescu, Algebraic surfaces, Universitext, Springer Verlag, New York, 2001.
\bibitem{B} T. Bauer, A simple proof for the existence of Zariski decompositions on surfaces, J. Alg. Geom. 18 (2009), 789 - 793.
\bibitem{Bh} P.B. Bhattacharya, The Hilbert function of two ideals, Proc. Camb. Phil. Soc. 53 (1957), 568 - 575.
\bibitem{BFJ} S. Bouksom, C. Favre and M. Jonsson, Differentiability of volumes and divisors and problem of Teissier, Journal of Algebraic Geometry 18 (2009), 279 - 308.
\bibitem{CGP} V. Cossart, C. Galindo and O. Piltant, Un exemple efectif de gradu\'e non noetherien assoici\'e une valuation diviorielle, Ann. Inst. Fouriier 50 (2000) 105 - 112.\bibitem{CJS} V. Cossart, U. Jannsen, S. Saito, Canonical  embedded and non-embedded resolution of singularities for excellent two dimensional schemes, arXiv:0905.2191.
\bibitem{C6} S.D. Cutkosky, On unique and almost unique factorization of complete ideals II, Inventiones Math. 98 (1989), 59-74.\bibitem{C1} S.D. Cutkosky, Multiplicities associated to graded families of ideals,  Algebra and Number Theory 7 (2013), 2059 - 2083.
\bibitem{C2} S.D. Cutkosky, Asymptotic multiplicities of graded families of ideals and linear series,  Advances in Mathematics 264 (2014), 55 - 113.
\bibitem{C4} S.D. Cutkosky, Teissier's problem on inequalities of nef divisors, J. Algebra Appl. 14 (2015).
\bibitem{C3} S.D. Cutkosky, Asymptotic Multiplicities, Journal of Algebra 442 (2015), 260 - 298.
\bibitem{C5} S.D. Cutkosky, Introduction to Algebraic Geometry, American Mathematical Society, Providence, RI,  2018.
\bibitem{CHR} S.D. Cutkosky, Ana Reguera and J\"urgen Herzog, Poincar\'e series of resolutions of surface singularities, Transactions of the American Math. Soc. 356 (2003), 1833 - 1874.
\bibitem{CSS} S.D. Cutkosky, Parangama Sarkar and Hema Srinivasan, Mixed multiplicities of filtrations, to appear in Transactions of the Amer. Math. Soc. 
\bibitem{CSV} S.D. Cutkosky, Hema Srinivasan and Jugal Verma, Positivity of Mixed Multipliciites of Filtrations, preprint
\bibitem{CS} S.D. Cutkosky and V. Srinivas, On a problem of Zariski on dimensions of linear systems, Annals Math. 137 (1993), 551 - 559.
\bibitem{D} J. Dugundji, Topology, Prentice Hall, 1966.
%\bibitem{ELS} L. Ein, R. Lazarsfeld and K. Smith, Uniform Approximation of Abhyankar valuation ideals in smooth function fields,
%Amer. J. Math. 125 (2003), 409 - 440
\bibitem{Fu} T. Fujita, Semipositive line bundles, J. Fac. Sci. Univ. Tokyo 30 (1983) 353 - 378.
%\bibitem{KG} K. Goel, Mixed Multiplicities of Ideals, Lecture Notes at IITB.
\bibitem{GGV} K. Goel, R.V. Gurjar and J.K. Verma, The Minkowski's equality and inequality for multiplicity of ideals, to appear in Contemporary Mathematics.
\bibitem{G1} A. Grothendieck and J. Dieudonn\'e, El\'ements de G\'eometrie Algebrique, EGA III, \'Etude cohomologique des faiseaux coh\'erents, Publ. Math. IHES 11 (1961).
% \bibitem{H} J. Huh, Milnor numbers of projective hypersurfaes and the chromatic polynomial of graphs, J. Amer. Math. Soc. 25 (2012), 907 - 927.
\bibitem{Ka} D. Katz, Note on multiplicity, Proc. Amer. Math. Soc. 104 (1988), 1021 - 1026.
\bibitem{KV} D. Katz and J. Verma, Extended Rees algebras and mixed multiplicities, Math. Z. 202 (1989), 111-128.
\bibitem{KK1} K. Kaveh and G. Khovanskii, Convex Bodies and Multiplicities of Ideals, Proc. Steklov Inst. Math. 286 (2014), 268 - 284.
\bibitem{KK} K. Kaveh and G. Khovanskii, Newton-Okounkov bodies, semigroups of integral points, graded algebras and intersection theory,  Annals of Math. 176 (2012), 925 - 978.
\bibitem{Kl} S. Kleiman, Toward a  numerical theory of ampleness, Ann. Math 84 (1966), 293 - 344.
\bibitem{LV1} R. Lazarsfeld,  Positivity in Algebraic Geometry Vol. 1, Springer Verlag, Berlin 2004. 
\bibitem{LM} R. Lazarsfeld and M. Musta\c{t}\u{a}, Convex bodies associated to linear series, Ann. Sci. Ec. Norm. Super 42 (2009) 783 - 835.
\bibitem{L1} J. Lipman, Rational singularities and applications to algebraic surfaces and unique factorization,
Publ. Math. IHES 36 (1969), 195 - 279.
\bibitem{L2} J. Lipman, Desingularization of 2-dimensional schemes, Annals of Math. 107 (1978), 115 - 207.
\bibitem{MS} H. Muhly and M. Sakuma, Asymptotic factorization of ideals, J. London Math. Soc. 38 (1963), 341 - 350.
\bibitem{Mus} M. Musta\c{t}\u{a}, On multiplicities of graded sequence of ideals, J. Algebra 256 (2002), 229-249.
\bibitem{Ok} A. Okounkov, Why would multiplicities be log-concave?, in The orbit method in geometry and physics, Progr. Math. 213, 2003, 329-347.
\bibitem{R} D. Rees, $\mathcal A$-transforms of local rings and a theorem on multiplicities of ideals, Proc. Cambridge Philos. Soc. 57 (1961), 8 - 17.
\bibitem{R1} D. Rees, Multiplicities, Hilbert functions and degree functions. In Commutative algebra: Durham 1981
(Durham 1981), London Math. Soc. Lecture Note Ser. 72, Cambridge, New York, Cambridge Univ. Press, 1982, 170 - 178.
\bibitem{R2} D. Rees, Valuations associated with a local ring II, J. London Math. Soc 31 (1956), 228 - 235.
\bibitem{R3} D. Rees, Izumi's theorem, in Commutative Algebra,  C. Huneke and J.D. Sally editors, Springer-Verlag 1989, 407 - 416.
\bibitem{RS} D. Rees and R. Sharp, On a Theorem of B. Teissier on Multiplicities of Ideals in Local Rings, J. London Math. Soc. 18 (1978), 449-463.
\bibitem{RM} J. Rotman, An introduction to homological algebra, second edition, Universitext, Springer, New York, 2009.
\bibitem{S} I. Swanson, Mixed multiplicities, joint reductions and a theorem of Rees, J. London Math. Soc. 48 (1993), 1 - 14.\bibitem{HS} I. Swanson and C. Huneke, Integral Closure of Ideals, Rings and Modules, Cambridge University Press, 2006.
\bibitem{T1} B. Teissier Cycles \'evanescents,sections planes et condition de Whitney, Singularti\'es \`a Carg\`ese 1972, Ast\'erique 7-8 (1973)
\bibitem{T2} B. Teissier, Sur une in\'egalit\'e pour les multipliciti\'es (Appendix to a paper by D. Eisenbud and H. Levine),
Ann. Math. 106 (1977), 38 - 44.
\bibitem{T3} B. Teissier, On a Minkowski type  inequality for multiplicities II, In C.P. Ramanujam - a tribute, Tata Inst. Fund. Res. Studies in Math. 8, Berlin - New York, Springer, 1978.
\bibitem{TV} N.V. Trung and J. Verma, Mixed multiplicities of ideals versus mixed volumes of polytopes, Trans. Amer. Math. Soc. 359 (2007), 4711 - 4727.
\bibitem{Z1} O. Zariski, Theory and Applications of Holomorphic Functions on Algebraic Varieties over Arbitrary Ground Fields, Memoirs of Amer. Math. Soc. New York (1951).
\bibitem{Z} O. Zariski, The theorem of Riemann-Roch for high multiples of an effective divisor on an algebraic surface. Ann. Math. 76 (1964), 560 - 615.
\end{thebibliography}
\end{document}